\documentclass[twoside]{amsart}
\usepackage{amsmath,amssymb}
\usepackage{amsthm}
\usepackage{tikz}
\tikzset{cross/.style={preaction={-,draw=white,line width=6pt}}}
\usepackage{amsfonts,mathrsfs}
\usepackage[]{enumitem}
\setlist[enumerate]{topsep=0pt, itemindent=*}
\usepackage{color}
\usepackage{graphicx}
\usepackage{datetime2}
\usepackage{refcount}
\usepackage{tcolorbox}
\tcbuselibrary{breakable}
\tcbuselibrary{skins}
\tcbset{enhanced,breakable}
\usepackage{bookmark}
\usepackage{xurl}
\hypersetup{unicode,bookmarksnumbered=true,hidelinks,final}
\usepackage{ulem}
\numberwithin{equation}{section}
\allowdisplaybreaks
\everymath{\displaystyle}
\usepackage{fancyhdr}
\theoremstyle{definition}
\newtheorem{definition}{Definition}[section]
\newtheorem{proposition}[definition]{Proposition}
\newtheorem{lemma}[definition]{Lemma}
\newtheorem{theorem}[definition]{Theorem}
\newtheorem{corollary}[definition]{Corollary}
\newtheorem{remark}[definition]{Remark}
\newtheorem{note}[definition]{Note}
\newtheorem{example}[definition]{Example}
\newtheorem{fact}[definition]{Fact}

\newtheorem{propdef}[definition]{Proposition-Definition}
\newtheorem{setting}[definition]{Setting}
\newtheorem*{proof*}{Proof}

\newcommand{\pd}[2]{\frac{\partial #1}{\partial #2}}

\newcommand{\Rn}[1]{\mathbb{R}^{#1}}
\newcommand{\Cn}[1]{\mathbb{C}^{#1}}

\newcommand{\pr}{\mathrm{pr}}

\newcommand{\id}{\mathrm{id}}
\newcommand{\iu}{\sqrt{-1}}
\newcommand{\Kahler }{K\"{a}hler}

\newcommand{\deflr}{\overset{\mathrm{def}}{\Leftrightarrow}}

\newcommand{\simarrow}{\xrightarrow{\sim}}

\DeclareMathOperator{\Aut}{\mathrm{Aut}}

\DeclareMathOperator{\Ker}{\mathrm{Ker}}

\DeclareMathOperator{\Sym}{\mathrm{Sym}}

\title{Hessian manifolds of positive constant Hessian sectional curvature}
\author{Hakobi Sakamoto}
\email{skmtha@ms.u-tokyo.ac.jp}
\thanks{This work was supported by JSPS KAKENHI Grant Number 26KJ1057.}
\begin{document}

\begin{abstract}
    Hessian structure on a manifold is a geometric structure 
    consisting of a pair~$(\nabla,g)$ of an affine connection~$\nabla$ 
    and a Riemannian metric~$g$ satisfying certain compatibility conditions. 
    It is also known as a dually flat structure in information geometry. 
    Hessian geometry is considered to be a real analogue of \Kahler~geometry; 
    in particular, Hessian sectional curvature corresponds to holomorphic 
    sectional curvature in \Kahler~geometry.
    The standard models of Hessian manifolds 
    with constant Hessian sectional curvature
    were classified by Furuhata and Kurose (2013) 
    in the nonpositive curvature case, 
    whereas the positive curvature case is open. 
    In this paper, we classify the standard models 
    in the positive curvature case.
    For this end, 
    we introduce Hessian structure on  
    quotient manifolds of \Kahler~manifolds 
    equipped with suitable Hamiltonian actions 
    of abelian Lie groups
    and realizing the standard models
    as such quotient manifolds.
\end{abstract}
\maketitle
\thispagestyle{fancy}

\tableofcontents
\clearpage


\section{Introduction}
A \textit{Hessian manifold} is a manifold equipped with a flat affine 
connection~$\nabla$ and a Riemannian metric~$g$ such that 
locally, in an affine chart of~$\nabla$,
the coefficients of~$g$ can be expressed 
as the Hessian matrix of a function~$\psi$, called a Hessian potential. 
It was defined by Shima (\cite{Shi76}) and 
is considered to be the real analogue of \Kahler~manifolds. 
Hessian manifolds are also known as \textit{dually flat manifolds} in 
\textit{information geometry}. They are studied as ``natural'' geometric structures
introduced on a class of statistical models known as \textit{exponential families}
(such as the normal, beta and categorical distributions).

Hessian sectional curvature was introduced in 
\cite{Shi95} as an analogue of holomorphic sectional curvature 
on a \Kahler~manifold.
The classification of standard models 
of Hessian manifolds with constant Hessian sectional curvature 
was established in the nonpositive curvature case in \cite{FK13}, 
while the positive curvature case remained open.
In this paper, we classify the standard models in the positive curvature case. 
In particular, including the nonpositive curvature case, 
we define a Hessian manifold as follows:
\begin{example}\label{exm:StandardModels}
    In what follows, we define a Hessian domain by an analytic
    potential function~$\psi$ on a domain in~$\mathbb{R}^n$.\\
    (1)~Let~$c<0$.\\
    \quad $\mathcal{S}^{n}(c):=\Rn{n}$,
    \quad $\psi:=
    -\frac{1}{c}\log(1+e^{\theta^1}+\cdots+e^{\theta^n}).$\\
    (2)~For~$k=0,\ldots, n$,\\
    \quad $\mathcal{R}^{n,k}(0):=\Rn{n}$,
    \quad $\psi:=
    \frac{1}{2}(\theta^{1})^2+\cdots+\frac{1}{2}(\theta^{k})^2
    +e^{\theta^{k+1}}+\cdots+e^{\theta^n}.$\\
    (3)~Let~$c>0$.\\
    (3-i)~$\mathcal{H}^{n}(c):=\left\{(\theta^1,\ldots,\theta^n)\in\Rn{n}~:~
    e^{\theta^1}+\cdots+e^{\theta^n}<1\right\}$,\\
    \quad $\psi:=-\frac{1}{c}\log(1-e^{\theta^1}-\cdots-e^{\theta^n}).$\\
    (3-ii)~$\mathcal{H}^{n,\mathrm{H}}(c):=\left\{(\theta^1,\ldots,\theta^n)\in
    \left(-\frac{\pi}{2},\frac{\pi}{2}\right)\times\Rn{n-1}~:~
    e^{\theta^2}+\cdots+e^{\theta^n}<\cos\theta^1\right\}$,\\
    \quad $\psi:=-\frac{1}{c}\log(\cos\theta^1-e^{\theta^2}-\cdots-e^{\theta^n}).$\\
    (3-iii)~For~$k=1,\ldots,n$,\\
    \quad $\mathcal{H}^{n,k}(c):=\left\{(\theta^1,\ldots,\theta^n)\in\Rn{n}~:~
    \theta^{1}+\frac{1}{2}(\theta^{2})^2+\cdots+\frac{1}{2}(\theta^{k})^2
    +e^{\theta^{k+1}}+\cdots+e^{\theta^n}<0\right\}$,\\
    \quad $\psi:=-\frac{1}{c}\log\left(-\theta^{1}-\frac{1}{2}(\theta^{2})^2-\cdots
    -\frac{1}{2}(\theta^{k})^2
    -e^{\theta^{k+1}}-\cdots-e^{\theta^n}\right).$
\end{example}
\begin{proposition}
    The above examples are 
    convex maximal Hessian domains of constant Hessian sectional curvature.
\end{proposition}
The proposition follows from~Propositions~\ref{prop:BeingStandardModels} and
\ref{prop:MaximalityOfModels}.
Then, we define as follows:
\begin{definition}\label{def:StandardModels}
    We call the Hessian manifolds defined in 
    Example~\ref{exm:StandardModels} 
    the \textit{standard models} of constant Hessian sectional curvature.
\end{definition}
The above definition is motivated by the following theorem:
\begin{theorem}\label{mthm:LocalEmbedding}
    For any connected and simply connected Hessian manifold
    with constant Hessian sectional curvature,
    there exists a local diffeomorphism preserving the Hessian structures
    to one of the standard models.
\end{theorem}
We note that, in the case~$c\leq 0$, 
Definition~\ref{def:StandardModels} and 
Theorem~\ref{mthm:LocalEmbedding}
coincide with the definition and statement, respectively, given in \cite{FK13}.
The following proposition is important in the proof:
\begin{proposition}\label{prop:QuotientHessian1}
    Let~$N$ be a~$2n$-dimensional \Kahler~manifold
    equipped with an effective, proper and Hamiltonian action
    of an~$n$-dimensional connected abelian Lie group~$\mathbb{A}$.
    Then, for~$N^\circ_{\mathbb{A}}:=\{p\in N~:~ \mathbb{A}_p=\{e\}\}$,
    the quotient manifold~$N^\circ_{\mathbb{A}}/\mathbb{A}$
    naturally carries a Hessian structure,
    where~$\mathbb{A}_p$ denotes the isotropy subgroup at~$p\in N.$
\end{proposition}
For a more precise statement, 
see Proposition-Definition~\ref{propdef:QuotientHessian}.
We then prove the following by classifying group actions on \Kahler~manifolds:
\begin{theorem}\label{mthm:Correspondence}
    In Proposition~\ref{prop:QuotientHessian1}, 
    if, in addition,~$N$ is connected, 
    simply connected, complete, 
    and has constant holomorphic sectional curvature, 
    then the quotient manifold~$N^\circ_{\mathbb{A}}/\mathbb{A}$ 
    is a convex maximal Hessian domain of constant Hessian sectional curvature.
    Conversely, every convex maximal Hessian domain 
    of constant Hessian sectional curvature can be obtained 
    as such a quotient, up to isomorphism.
\end{theorem}

Here, when~$c=-1, 0, 1$, all the standard models
are derived from the
simple quadratic natural exponential families
defined in \cite{Cas96}.

The setting of Proposition~\ref{prop:QuotientHessian1} 
can be viewed as a certain generalization of toric \Kahler~geometry.
Roughly speaking, a toric \Kahler~manifold is a \Kahler~manifold 
equipped with a Hamiltonian action of 
a torus of half the dimension of the manifold. 
The relationship with Hessian geometry 
through the Hessian structure on the image of the moment map 
is discussed in detail in 
Guillemin~\cite{Gui94}, Abreu~\cite{Abr03} and Molitor~\cite{Mol21}.

The organization of this paper is as follows.
In Section 2, we investigate the basic properties 
of the standard models. We also reduce Theorem~\ref{mthm:LocalEmbedding} 
to a local problem (Proposition~\ref{prop:LocalHessianEmbedding}).
In Section 3, we introduce a natural Hessian structure 
on quotients of \Kahler~manifolds 
by suitable group actions (Section~\ref{sec:QuotientHessian}) 
and investigate its basic properties.
In Section 4, we realize the standard models 
as quotients of suitable \Kahler~manifolds 
by group actions and prove Proposition~\ref{prop:LocalHessianEmbedding}. 
Furthermore, we prove Theorem~\ref{mthm:Correspondence}.

\section{Hessian manifolds of constant Hessian sectional curvature}
In this section, we investigate the basic properties 
of the standard models given in Example~\ref{exm:StandardModels}.
We also reduce Theorem~\ref{mthm:LocalEmbedding} 
to a local problem (Proposition~\ref{prop:LocalHessianEmbedding}),
which is proved using results from Sections~\ref{Sec:QuotientHessian}
and~\ref{Sec:Correspondence}.

\subsection{Hessian manifold}\label{sec:StandardModels}
In this section, we review the definition of Hessian manifolds.
In the following, let~$M$ be an~$n$-dimensional manifold,
$\nabla$ be an affine connection on~$M$,
and~$g$ be a Riemannian metric~$g$ on~$M$.

\begin{definition}[\cite{Shi07}]
    (1)
    A triple~$(M, \nabla,g)$ is called a \textit{Hessian manifold}
    if~$N$ is flat, that is, its curvature and torsion both vanish,
    and, around every point~$p\in M$,
    there exists a smooth function~$\psi$ such that
    $g=\nabla^2 \psi$.
    We call~$\psi$ a \textit{(Hessian) potential} of~$g$.\\
    (2)~A Hessian manifold~$(U,\nabla,g)$
    is called a \textit{Hessian domain}
    if~$U$ is a domain, that is, a connected open subset
    of~$\Rn{n}$,
    $\nabla$ coincides with the standard flat connection on~$\Rn{n}$,
    and~$g$ admits a globally defined Hessian potential~$\psi$ on~$U$.
\end{definition}
\begin{fact}[{\cite[Proposition 2.1]{Shi07}}]\label{fac:HessianCondition}
    A pair~$(\nabla, g)$ defines a Hessian structure
    if and only if~$\nabla g$ is symmetric in its three arguments.
    In particular, this is equivalent to
    \begin{equation*}
        Xg(Y,Z)-Yg(X,Z)=0
    \end{equation*}
    for any locally defined parallel vector fields
    $X, Y, Z$ with respect to~$\nabla$.
\end{fact}
\begin{note}\label{not:HessianCondition}
    For an open set~$U\subset \Rn{n}$, 
    the pair consisting of the standard flat 
    connection~$\nabla$ on~$\Rn{n}$ 
    and a Riemannian metric~$g$ defines 
    a Hessian structure if and only if, 
    with respect to the standard chart~$(\theta^1,\ldots,\theta^n)$, 
    the following condition holds for all~$i,j,k=1,\ldots,n$:
    \begin{equation}\label{eq:HessianCondition}
        \pd{g_{jk}}{\theta^i}=\pd{g_{ik}}{\theta^j}.
    \end{equation}
\end{note}
\begin{definition}
    We denote by~$\mathfrak{X}(M)$ the vector space of
    all vector fields on~$M$.
    For an affine connection~$\nabla$ on~$M$,
    we define~$\mathfrak{X}_{\nabla}(M)
    :=\{X\in\mathfrak{X}(M)~:~\nabla X=0\}$.
    In particular, if~$\nabla$ is torsion-free, then
    $\mathfrak{X}_{\nabla}(M)$ is an abelian Lie subalgebra of
    $\mathfrak{X}(M)$ with respect to the Lie bracket.
\end{definition}

In what follows, when we simply write 
``Hessian domain~$(U,\nabla, g=\nabla^2\psi)$'',
we mean that~$U\subset \Rn{n}$ is a domain,
$\nabla$ is the standard flat connection on~$\Rn{n}$,
and~$\psi\colon U\rightarrow \Rn{}$ is a Hessian potential for~$g$.
For a Hessian domain, we denote by~$(\theta^1,\ldots,\theta^n)$
the standard coordinate system on~$\Rn{n}$.
\begin{note}
    Let~$(U,\nabla, g=\nabla^2\psi)$ be a Hessian domain.
    Then, the \textit{gradient map}:
    \begin{equation*}
        \nabla\psi:=\left(\pd{\psi}{\theta^1},\ldots, \pd{\psi}{\theta^n}\right)
        \colon U\rightarrow \Rn{n}
    \end{equation*}
    is locally diffeomorphic and, if~$U$ is convex, 
    it is moreover a diffeomorphism.
\end{note}
\begin{definition}\label{def:DualChart}
    Let~$(U,\nabla, g=\nabla^2 \psi)$ be a convex Hessian domain.
    Then we define the 
    \textit{dual chart}~$(\eta_1,\ldots, \eta_n)$
    of the standard chart~$(\theta^1,\ldots, \theta^n)$
    by
    $\eta_i:=\pd{\psi}{\theta^i}$ ($i=1,\ldots, n$).
\end{definition}

\subsection{Hessian sectional curvature and variance function}
In this section, we review the definition of
Hessian sectional curvature.
Then, we introduce a variance function
for a Hessian domain and 
investigate the relation to Hessian sectional curvature.
Moreover, we prove that
the standard models defined in Example~\ref{exm:StandardModels}
are Hessian domains with constant Hessian sectional curvature.
Throughout the following, 
let~$M$ be an~$n$-dimensional manifold, 
$\nabla$ an affine connection on~$M$, 
and~$g$ a Riemannian metric on~$M$.
For simplicity, we use subscripts to denote the partial derivatives 
of a Hessian potential~$\psi$, as in
$\psi_{ij}:=\frac{\partial^2\psi}{\partial\theta^i\partial\theta^j}$.

\begin{definition}[\cite{Shi07}]
    Let~$(M,\nabla,g)$ be a Hessian manifold.\\
    (1)~The \textit{difference tensor}~$\gamma$ is the~$(1,2)$-tensor field 
    defined by~$\gamma:=\nabla^g-\nabla$,
    where~$\nabla^g$ denotes the Levi-Civita connection of~$g$.\\
    (2)~The \textit{Hessian curvature tensor}~$Q$
    is the~$(1,3)$-tensor field defined by
    \begin{equation*}
        Q(X,Y,Z):=(\nabla_Y\gamma)(X,Z)
    \end{equation*}
    for any vector fields~$X, Y, Z$ on~$M$.\\
    (3)~Let~$S^2(TM)$ denote the symmetric square of~$TM$.
    For any~$p\in M$ and~$X,Y\in T_pM\setminus\{0\}$,
    the \textit{Hessian sectional curvature}~$q:S^2(TM)\rightarrow \Rn{}$
    is defined by
    \begin{equation*}
        q(X\odot Y):=
        \frac{2g(X,Q(X,Y,Y))}
        {g(X,X)g(Y,Y)+g(X,Y)^2}.
    \end{equation*}
\end{definition}
\begin{note}\label{not:HessianCurvature}
    For a Hessian manifold~$(M,\nabla,g)$, 
    the Hessian curvature tensor~$Q$ s
    atisfies the following identity with respect to any 
    affine chart~$(\theta^1,\ldots,\theta^n)$ 
    and any locally defined Hessian potential~$\psi$:
    \begin{equation*}
        Q_{iajb}:=g\left(\pd{}{\theta^i},
        Q\left(\pd{}{\theta^a},\pd{}{\theta^j},
        \pd{}{\theta^b}\right)\right)
        =\frac12\psi_{ijab}-\frac12g^{kl}\psi_{ijk}\psi_{abl}.
    \end{equation*}
    Here,~$(g^{ij})$ denotes the inverse matrix of~$(g_{ij})$.
\end{note}
\begin{fact}[\cite{Shi07}]
    A Hessian manifold has
    constant Hessian sectional curvature~$c\in \Rn{}$
    if and only if 
    \begin{equation*}
        Q_{iajb}=g\left(\pd{}{\theta^i},
        Q\left(\pd{}{\theta^a},\pd{}{\theta^j},
        \pd{}{\theta^b}\right)\right)=\frac{c}{2}(g_{ia}g_{jb}+g_{ib}g_{ja})
    \end{equation*}
    in any affine chart~$(\theta^1,\ldots,\theta^n)$.
\end{fact}

We introduce a variance function for convex Hessian domains.
\begin{definition}
    Let~$(U,\nabla, g=\nabla^2\psi)$ be a convex Hessian domain,
    $\nabla \psi\colon U\rightarrow \Rn{n}$ be
    the gradient map,
    and~$(\eta_1,\ldots, \eta_n)$ be the dual chart
    (Definition~\ref{def:DualChart}).\\
    (1)~
    The \textit{variance function}~$V\colon (\nabla\psi)(U)\rightarrow
    \Sym(n,\Rn{})$ is defined by:
    \begin{equation*}
        V_{ij}:=\frac{\partial^2\psi}{\partial\theta^i\partial\theta^j}
        \circ (\nabla\psi)^{-1}.
    \end{equation*}
    (2)~The variance function~$V$ is 
    said to be \textit{simple quadratic}
    if there exists some~$c\in\Rn{}$ such that
    \begin{equation*}
        \frac{\partial^2V_{ij}}{\partial\eta_k\partial\eta_l}
        =\frac{c}{2}(\delta_i^k\delta_j^l+\delta_i^l\delta_j^k)
    \end{equation*}
    for all~$i,j,k,l$.
    Here,~$\delta_i^j$ denotes the Kronecker delta.
\end{definition}
\begin{note}\label{not:SimpleQuadratic}
    The variance function~$V$ is simple quadratic
    if and only if all the components~$V_{ij}$
    are polynomials of degree at most two,
    whose quadratic part is~$c\eta_i\eta_j$.
\end{note}
\begin{remark}
    The variance function defined above coincides 
    with the variance function of a natural exponential family 
    (NEF) when the Hessian domain is determined by an NEF.
    In this case, the notion of ``simple quadratic'' 
    coincides with that in~\cite{Cas96}.
\end{remark}

The variance function is related to the Hessian curvature tensor.
\begin{lemma}\label{lem:CurvatureVariance}
    Let~$(U,\nabla,g=\nabla^2\psi)$ be a convex Hessian domain.
    Then, for the Hessian curvature tensor~$Q$
    and the variance function~$V$,
    the following holds:
    for all~$i,a,j,b$,
    \begin{equation*}
        Q_{iajb}=\frac12 g_{ak}g_{bl}
        \left(\frac{\partial^2V_{ij}}{\partial\eta_k\partial\eta_l}\circ(\nabla\psi)
        \right).
    \end{equation*}
    Here, we set~$Q_{iajb}=g\left(\pd{}{\theta^i},
        Q\left(\pd{}{\theta^a},\pd{}{\theta^j},
        \pd{}{\theta^b}\right)\right)$.
\end{lemma}
\begin{proof}
    We denote by~$(g^{ij})$ the inverse matrix of~$(g_{ij})$.
    Since~$\pd{\eta_a}{\theta^b}=\psi_{ab}=g_{ab}$,
    we have~$\pd{\theta^a}{\eta_b}=g^{ab}$.
    Thus,
    \begin{align*}
        &\frac12 g_{ak}g_{bl}
        \left(\frac{\partial^2V_{ij}}{\partial\eta_k\partial\eta_l}\circ(\nabla\psi)
        \right)\\
        &=\frac12 g_{ak}g_{bl}\pd{\theta^p}{\eta_k}\pd{}{\theta^p}
        \left(\pd{\theta^q}{\eta_l}\pd{}{\theta^q}\psi_{ij}
        \right)
        =\frac12 g_{bl}
        \left(\pd{g^{ql}}{\theta^a}\psi_{ijq}
        +g^{ql}\psi_{ijaq}
        \right)\\
        &=-\frac12\pd{g_{bl}}{\theta^a}g^{ql}\psi_{ijq}
        +\frac12\psi_{ijab}
        =\frac12\psi_{ijab}-\frac12g^{ql}\psi_{ijq}\psi_{abl}
        =Q_{iajb}.
    \end{align*}
    In the last equality, we have used Note~\ref{not:HessianCurvature}.
\end{proof}
\begin{proposition}\label{prop:CurvatureVariance}
    A Hessian domain~$(U,\nabla, g=\nabla^2\psi)$ has 
    constant Hessian sectional curvature~$c\in\Rn{}$,
    if and only if its variance function~$V$ is simple quadratic.
\end{proposition}
\begin{proof}
    Since the argument is local, 
    we may assume that~$U$ is convex, 
    and the result follows immediately from Lemma~\ref{lem:CurvatureVariance}.
\end{proof}

Starting from a Hessian domain 
with constant Hessian sectional curvature,
we can construct another one of one higher dimension.
\begin{proposition}\label{prop:OneMoreDimension}
    Let~$(U,\nabla, g=\nabla^2\psi)$ be a
    Hessian domain with constant Hessian sectional curvature
    $c\in\Rn{}$.\\
    Then, the domain~$\tilde{U}:=\begin{cases}
        U\times \Rn{}& (c< 0)\\
        U\times \Rn{}& (c=0)\\
        \{(\theta,t)\subset U\times \Rn{}~:~
        e^{-c\psi}-e^{t}>0\}&(c>0)
    \end{cases}$
    together with the function
    $\tilde{\psi}:=\begin{cases}
        \frac{1}{-c}\log(e^{-c\psi}+e^{t}) & (c< 0)\\
            \psi+e^{t}&(c=0)\\
            \frac{1}{-c}\log(e^{-c\psi}-e^{t}) & (c> 0)
        \end{cases}$
    defines another
    Hessian domain 
    $(\tilde{U}, \tilde{\nabla}, \tilde{g}:=\tilde{\nabla}^2\tilde{\psi})$
    with constant Hessian sectional curvature~$c$,
    where~$\tilde{\nabla}$ denotes the standard flat connection on~$\Rn{n+1}$.
\end{proposition}
\begin{proof}
    We first consider the case~$c=0$.
    In this case, 
    $(\tilde{U},\tilde{\nabla},\tilde{g})$ 
    is the product of the Hessian domains
    $(U,\nabla,g)$ and~$(\Rn{},\nabla,e^t dt^2)$.
    By assumption,~$(U,\nabla,g)$ has constant Hessian sectional curvature~$0$, 
    and a direct computation shows that~$(\Rn{},\nabla,e^t dt^2)$ 
    also has constant Hessian sectional curvature~$0$.
    Hence, the assertion follows.

    We therefore assume~$c\neq 0$.
    Let~$s:=\mathrm{sgn}~c=\frac{c}{|c|}$.
    First, we show that
    $(\tilde{U}, \tilde{\nabla},\tilde{g}=\tilde{\nabla}^2\tilde{\psi})$
    is a Hessian domain.
    It suffices to show that
    $\tilde{g}$ is positive definite.
    With respect to the standard chart
    $(\theta^1,\ldots,\theta^n, t)$ on~$\Rn{n+1}$,
    the Hessian matrix of~$\tilde{\psi}$
    is given by
    \begin{align*}
        (\tilde{\nabla}^2\tilde{\psi})_{ij}
        &=\begin{pmatrix}
            \frac{\psi_{ij}}{1-se^{t+c\psi}}
            +\frac{|c| e^{t+c\psi}\psi_i\psi_j}{(1-se^{t+c\psi})^2}&
            \frac{se^{t+c\psi}\psi_i}{(1-se^{t+c\psi})^2}\\
            \frac{se^{t+c\psi}\psi_j}{(1-se^{t+c\psi})^2}&
            \frac{e^{t+c\psi}}{|c|(1-se^{t+c\psi})^2}
        \end{pmatrix}\\
        &=\begin{pmatrix}
            \delta_i^a & c\psi_i\\
            0 &1
        \end{pmatrix}
        \begin{pmatrix}
            \frac{\psi_{ab}}{1-se^{t+c\psi}}
            & 0\\
            0 &
            \frac{e^{t+c\psi}}{|c|(1-se^{t+c\psi})^2}
        \end{pmatrix}
        \begin{pmatrix}
            \delta_j^b & 0\\
            c\psi_j &1
        \end{pmatrix}.
    \end{align*}
    Thus,~$\tilde{g}$ is positive definite.
    
    Next, we show the resulting Hessian domain~$\tilde{U}$ has
    constant Hessian sectional curvature~$c$.
    Since the argument is local, 
    we may assume that~$U$ and~$\tilde{U}$ are convex.
    With respect to the dual chart
    $(\tilde{\eta}_1,\ldots,\tilde{\eta}_n,\tilde{\eta}_t):=
    \left(\pd{\tilde{\psi}}{\theta^1},\ldots,
    \pd{\tilde{\psi}}{\theta^n},\pd{\tilde{\psi}}{t}\right)$
    of~$(\theta^1,\ldots,\theta^n,t)$,
    we have
    \begin{equation*}
        \tilde{\eta}_i=\frac{\psi_i}{1-se^{t+c\psi}},
        \quad \tilde{\eta}_t=\frac{e^{t+c\psi}}{|c|(1-se^{t+c\psi})}
    \end{equation*}
    and therefore~$\eta_i=\frac{\tilde{\eta}_i}{1+c\tilde{\eta}_t}$.
    Thus,
    the variance functions~$V, \tilde{V}$ of~$\psi, \tilde{\psi}$,
    respectively, satisfy
    \begin{align*}
        \tilde{V}(\tilde{\eta})
        &=\begin{pmatrix}
            \frac{\psi_{ij}}{1-se^{t+c\psi}}
            +\frac{|c| e^{t+c\psi}\psi_i\psi_j}{(1-se^{t+c\psi})^2}&
            \frac{se^{t+c\psi}\psi_i}{(1-se^{t+c\psi})^2}\\
            \frac{se^{t+c\psi}\psi_j}{(1-se^{t+c\psi})^2}&
            \frac{e^{t+c\psi}}{|c|(1-se^{t+c\psi})^2}
        \end{pmatrix}\\
        &=c\begin{pmatrix}
            \tilde{\eta}_i\tilde{\eta}_j 
            &\tilde{\eta}_i\tilde{\eta}_t\\
            \tilde{\eta}_j\tilde{\eta}_t
            &\tilde{\eta}_t^2
        \end{pmatrix}
        +\begin{pmatrix}
            (1+c\tilde{\eta}_t)(V_{ij}(\eta)-c\eta_i\eta_j)
            &0\\
            0& \tilde{\eta}_t
        \end{pmatrix}.
    \end{align*}
    It follows from Note~\ref{not:SimpleQuadratic} that
    they are simple quadratic. 
    The claim follows from Proposition~\ref{prop:CurvatureVariance}.
\end{proof}
\begin{remark}
    When~$c\leq0$, we have~$\{(\theta,t)\subset U\times \Rn{}~:~
    e^{-c\psi}-se^{t}>0\}=U\times \Rn{}$.
    Moreover, the limit of~$\frac{1}{-c}\log(e^{-c\psi}-se^{t}) $
    as~$c\rightarrow 0$
    coincides with~$\psi+e^{t}$.
\end{remark}

We prove that
the standard models defined in Example~\ref{exm:StandardModels}
are Hessian domains with constant Hessian sectional curvature.
\begin{proposition}\label{prop:BeingStandardModels}
    Let~$c\in \Rn{}$. \\
    All the standard models
    $\begin{cases}
        \mathcal{S}^n(c) & (c<0)\\
        \mathcal{R}^{n,k}(c)~(k=0,\ldots, n)& (c=0)\\
        \mathcal{H}^n(c), \mathcal{H}^{n,\mathrm{H}}(c),
    \mathcal{H}^{n,k}(c)~(k=1\ldots, n) & (c>0)
    \end{cases}$
    defined in 
    Example~\ref{exm:StandardModels}
    have constant Hessian sectional curvature~$c$.
\end{proposition}
\begin{proof}
    When~$c<0$,
    $\mathcal{S}^1(c)$ has constant Hessian sectional curvature
    $c$ by direct computation.
    Then, by repeatedly applying Proposition~\ref{prop:OneMoreDimension},
    the claim follows for any~$n$.

    When~$c=0$,
    for any~$k\in\mathbb{Z}_{\geq 0}$,
    $\mathcal{R}^{k,k}(0)$ has
    constant Hessian sectional curvature~$0$
    by direct computation.
    Then, by repeatedly applying Proposition~\ref{prop:OneMoreDimension},
    the claim follows for any~$n$ and~$k=0,\ldots, n$.

    When~$c>0$, for any~$n$,
    $\mathcal{H}^{n}(c)$ and~$\mathcal{H}^{n,\mathrm{H}}(c)$
    has constant Hessian sectional curvature~$c$ by direct computation.
    Moreover, for any~$n$,
    it was proved in \cite[Theorem 7]{Shi95}
    that 
    $\mathcal{H}^{n,n}(c)$ has constant Hessian sectional curvature
    $c$.
    Thus, by repeatedly applying Proposition~\ref{prop:OneMoreDimension},
    the claim follows.
\end{proof}

\subsection{Maximality of Hessian domains}\label{sec:ConvexMaximal}
In this section, we discuss the maximality of Hessian domains.
We also show that the standard models (Example~\ref{exm:StandardModels}) 
are convex maximal Hessian domains.
In the following,~$\nabla$ denotes the standard flat connection
on~$\Rn{n}$.

\begin{definition}\label{def:Maximality}
    (1)~For a pair~$(U,f)$ consisting of
    a domain~$U\subset \Rn{n}$ and an analytic function~$f\colon U\rightarrow\Rn{}$,
    we define an order relation~$\leq$ by
    \begin{equation*}
        (U,f)\leq (U',f')\deflr
        U\subset U'~\mathrm{and}~f'|_U=f.
    \end{equation*}
    (2)~For a pair~$(U,g)$ consisting of
    a domain~$U\subset \Rn{n}$ and an analytic Riemannian metric~$g$ on~$U$,
    we define an order relation~$\leq$ by
    \begin{equation*}
        (U,g)\leq (U',g')\deflr
        U\subset U'~\mathrm{and}~g'|_U=g.
    \end{equation*}
    (3)~A Hessian domain~$(U,\nabla, g)$ is said to be
    \textit{maximal} if
    $(U, g)$ is maximal with respect to~$\leq$ in (2).
\end{definition}
\begin{note}\label{not:DomainSum}
    Let~$U,U'$ be domains and~$f, f'$ be analytic 
    functions on~$U,U'$, respectively.
    Suppose that~$U\cap U'$ is connected and 
    that~$f=f'$ on~$U\cap U'$.
    Then,~$U\cup U'$ is a domain, that is a connected open subset
    of~$\Rn{n}$,
    and
    $(U,f), (U', f') \leq (U\cup U', \tilde{f})$
    where
    $\tilde{f}\colon U\cup U'\rightarrow\Rn{}$
    is defined by
    $\tilde{f}:=f$ on~$U$ and~$\tilde{f}:=f'$ on~$U'$.
    The same applies to Riemannian 
    metrics~$g$ and~$g'$ on~$U$ and~$U'$, respectively.
\end{note}
\begin{note}\label{not:PotentialAnaliticity}
    For a Hessian domain~$(U,\nabla, g)$
    whose Riemannian metric~$g$ is analytic, 
    there exists locally an analytic Hessian potential
    of~$g$, by applying the real analytic version of the Poincar\'{e} lemma
    to the argument in \cite[Proposition 2.1]{Shi07}.
\end{note}
\begin{lemma}\label{lem:Maximalities}
    For a convex Hessian domain~$(U,\nabla, g=\nabla^2\psi)$
    whose Riemannian metric~$g$ and Hessian potential~$\psi$
    are analytic,
    if~$(U,\psi)$ is maximal,
    then~$(U,g)$ is also maximal
    with respect to the order relations~$\leq$ 
    in Definition~\ref{def:Maximality}.
\end{lemma}
\begin{proof}
    Suppose that a pair~$(U', g')$ of
    a domain~$U'\subset \Rn{n}$
    and an analytic Riemannian metric~$g'$ on~$U'$,
    satisfies~$(U,\nabla^2\psi)\leq (U', g')$.
    Since~$(\nabla,g')$ defines a Hessian structure on~$U$, 
    the analyticity implies that 
    the identity~\eqref{eq:HessianCondition} 
    in Note~\ref{not:HessianCondition} also holds on~$U'$.
    Hence,~$(\nabla,g')$ defines a Hessian structure on~$U'$ as well.
    It suffices to show that~$U$ is 
    a closed subset of~$U'$,
    equivalently, that~$\overline{U}\cap U'=U$.
    The inclusion~$U\subset \overline{U}\cap U'$ is clear.
    It remains to show that~$\overline{U}\cap U'\subset U$.
    Let~$p\in \overline{U}\cap U'$,
    and let
    $V\subset \Rn{n}$ be an open convex neighborhood of~$p$ 
    with~$V\subset U'$,
    such that an analytic Hessian potential~$\psi'$ of~$g'$
    exists on~$V$.
    On~$U\cap V$, we have~$\nabla^2\psi=g'$, and hence
    $\psi'-\psi$ is an affine function.
    Thus, by replacing~$\psi'$ if necessary, we may assume that
    $\psi'=\psi$.
    Since~$U$ and~$V$ are convex, so is~$U\cap V$.
    Therefore, by Note~\ref{not:DomainSum}, 
    $\psi$ extends to~$U\cup V$.
    By maximality, we must have~$U\cup V=U$, and in particular,
    $p\in U$.
\end{proof}

\begin{lemma}\label{lem:LogMaximality}
    Let~$f\colon \Rn{n}\rightarrow\Rn{}$ be an analytic function
    and~$U$ be one of the connected components of~$f^{-1}(\Rn{}_{>0})$.
    Then,~$(U,\log f)$ is maximal
    with respect to the order relations~$\leq$ 
    in Definition~\ref{def:Maximality}.
\end{lemma}
\begin{proof}
    Suppose that a pair~$(U', f')$ of
    a domain~$U'\subset \Rn{n}$
    and an analytic function~$f'$ on~$U'$,
    satisfies~$(U,\log f)\leq (U',f')$.
    Then~$f=e^{f'}>0$ on~$U$,
    and hence, by analiticity, also on~$U'$.
    Thus~$U'\subset f^{-1}(\Rn{}_{>0})$
    and
    $(U', f')=(U,\log f)$.
\end{proof}
\begin{proposition}\label{prop:MaximalityOfModels}
    The standard models (Example~\ref{exm:StandardModels})
    are convex and maximal.
\end{proposition}
\begin{proof}
    Let the Hessian domain~$(U,\nabla,g=\nabla^2\psi)$
    be one of the standard models, 
    where~$\psi$ is the Hessian potential defined in 
    Example~\ref{exm:StandardModels}. 
    The analyticity of~$\psi$ is 
    clear by definition. 
    Moreover, since~$U$ is a level set of a convex function, 
    it is a convex domain.
    We next show the maximality. 
    If~$c\leq 0$, then~$U=\mathbb{R}^n$, 
    and hence~$(U,\nabla,g)$ 
    is a maximal Hessian domain. 
    Suppose that~$c>0$. 
    Then~$f:=e^\psi$ extends to 
    a function~$f\colon\mathbb{R}^n\to\mathbb{R}$. 
    In either case,~$U$ 
    is a connected component 
    of~$f^{-1}(\mathbb{R}_{>0})$. 
    Thus, by Lemma~\ref{lem:LogMaximality}, 
    $(U,\psi)$ is maximal, and consequently, 
    by Lemma~\ref{lem:Maximalities}, so is~$(U,\nabla^2\psi)$.
\end{proof}

\subsection{Embedding into standard models}
\label{sec:HessianEmbedding}
In this section, we reduce Theorem~\ref{mthm:LocalEmbedding} 
to a local problem (Proposition~\ref{prop:LocalHessianEmbedding}).

\begin{definition}[\cite{Tu17}]
    Let~$M$ and~$M'$ be manifolds equipped with affine connections
    $\nabla$ and~$\nabla'$, respectively.
    A local diffeomorphism~$f\colon M\rightarrow M'$ is said to be
     \textit{connection-preserving}
    if it satisfies~$f_*(\nabla_XY)=\nabla'_{f_*X}(f_*Y)$
    for any locally defined vector fields~$X$ and~$Y$ on~$M$.
\end{definition}
\begin{definition}
    Let~$(M,\nabla,g)$ and~$(M',\nabla',g')$
    be Hessian manifolds.
    A local diffeomorphism~$f\colon M\rightarrow M'$ is said to 
    \textit{preserve the
    Hessian structure} if
    it is a connection-preserving isometry,
    that is,~$f$ satisfies~$f^*g'=g$.
\end{definition}
\begin{note}\label{not:AffineMapParallelVector}
    Let~$M$ and~$M'$ be manifolds equipped with affine connections
    $\nabla$ and~$\nabla'$, respectively, 
    and let~$f\colon  M\rightarrow M'$ be a
    diffeomorphism.
    Then~$f$ is connection-preserving if and only if,
    $f$ sends any locally defined 
    parallel vector field on~$M$
    to another one on~$M'$.
    In particular, if~$f$ is connection-preserving, 
    then~$f_*X\in\mathfrak{X}_{\nabla'}(M')$ 
    for any~$X\in\mathfrak{X}_{\nabla}(M)$.
\end{note}
\begin{fact}[\cite{KN63}]\label{fac:AffineExtension}
    Let~$\nabla, \nabla'$ be analytic affine connections
    on analytic manifolds
    $M, M'$, respectively.
    Suppose that~$M$ is connected and simply connected,
    and that~$\nabla'$ is complete.
    Then, for any connected open subset~$V\subset M$,
    every connection-preserving local diffeomorphism
    $f\colon V\rightarrow M'$
    uniquely extends to a 
    connection-preserving local diffeomorphism~$\tilde{f}\colon M\rightarrow M'$.
\end{fact}
\begin{proposition}\label{prop:HessianExtension}
    Let~$(M,\nabla',g')$ be
    an~$n$-dimensional connected and simply connected Hessian manifold,
    and let~$(U,\nabla, g=\nabla^2\psi)$
    be an~$n$-dimensional convex maximal Hessian domain.
    Suppose that the manifold~$M$, 
    the affine connections~$\nabla', \nabla$,
    the Riemannian metric~$g', g$,
    and the Hessian potential~$\psi$ are all analytic.
    Then, for any connected open subset~$V\subset M$,
    every local diffeomorphism
    $f\colon V\rightarrow U$ preserving the Hessian structures
    uniquely extends to a 
    local isomorphism~$\tilde{f}\colon M\rightarrow U$
    preserving the Hessian structures.
\end{proposition}
\begin{proof}
    By Fact~\ref{fac:AffineExtension} and 
    analiticity,
    $f$ extends to a connection-preserving
    local diffeomorphism~$\tilde{f}\colon M\rightarrow \Rn{n}.$
    Then,~$\tilde{f}(M)\subset \Rn{n}$ is a connected open subset.

    It suffices to show that~$\tilde{f}(M)\cap \overline{U}\subset U$.
    Indeed, since
    \begin{equation*}
        \tilde{f}(M)=(\tilde{f}(M)\cap\overline{U})\sqcup
    (\tilde{f}(M)\cap \Rn{n}\setminus\overline{U})
    \subset (\tilde{f}(M)\cap U)\sqcup
    (\tilde{f}(M)\cap \Rn{n}\setminus\overline{U}),
    \end{equation*}
    and since~$\tilde{f}(M)\cap U\neq \emptyset$
    and~$\tilde{f}(M)$ is connected,
    it follows that~$\tilde{f}(M)\subset U$.
    
    We show that~$\tilde{f}(M)\cap \overline{U}\subset U$.
    Let~$p\in\tilde{f}(M)\cap \overline{U}$, and
    suppose that for a convex open neighborhood~$V\subset \tilde{f}(M)$
    of~$p$
    and an open subset~$W\subset M$,
    $f|_W: W\rightarrow V$ is a diffeomorphism.
    Then the Riemannian metric~$g'':=(f|_W^{-1})^*g'$
    on~$V$ satisfies~$g''=(f|_W^{-1})^*g'=(f|_W^{-1})^*f^*g=g$
    on~$V\cap U$.
    Thus,~$g$ extends to~$U\cup V$, 
    and hence~$U\cup V\subset U$ by the maximality of~$U$.
    In particular,~$p\in U$.
\end{proof}
\begin{corollary}\label{cor:ConvexMaximalIsomorphism}
    Let~$(U,\nabla, g=\nabla^2\psi),~(U',\nabla', g'=(\nabla')^2\psi')$ be
    $n$-dimensional convex maximal Hessian domains
    whose Riemannian metrics are analytic.
    Then, for any connected open subset~$V\subset U$,
    every local diffeomorphism
    $f\colon V\rightarrow U'$ preserving the Hessian structures
    uniquely extends to a 
    diffeomorphism~$\tilde{f}\colon U\rightarrow U'$
    preserving Hessian structures.
\end{corollary}
\begin{proof}
    By Proposition~\ref{prop:HessianExtension},
    $f$ extends to a local isomorphism~$\tilde{f}\colon U\rightarrow U'$
    preserving the Hessian structures.
    Since~$U$ is convex, it follows that~$\tilde{f}$ is injective.
    Indeed, if~$\tilde{f}$ maps 
    two distinct points~$p, q\in U$ 
    to the same point in~$U'$, 
    then, by convexity, 
    $\tilde{f}$ also maps the line segment joining~$p$ and~$q$ 
    to a single point in~$U'$, 
    contradicting the fact that $\tilde{f}$ is a local diffeomorphism.
    Therefore,~$\tilde{f}\colon U\rightarrow \tilde{f}(U)\subset U'$ 
    is an isomorphism of Hessian manifolds. 
    Since~$(U,g)$ is maximal with respect to the order~$\leq$ 
    in Definition~\ref{def:Maximality}, 
    so is~$(\tilde{f}(U),g')$. 
    Hence~$\tilde{f}(U)=U'$, 
    and thus~$\tilde{f}\colon U\rightarrow U'$ 
    is an isomorphism of Hessian manifolds.
\end{proof}

To prove Theorem~\ref{mthm:LocalEmbedding},
it suffices to prove the following
proposition:
\begin{proposition}\label{prop:LocalHessianEmbedding}
    Let~$M$ be a Hessian manifold with 
    constant Hessian sectional curvature~$c\in\Rn{}$.
    Then, for any~$p\in M$, 
    there exists 
    an open neighborhood~$U\subset M$ of~$p$ and
    one of the standard models~$\tilde{M}$
    with constant Hessian sectional curvature~$c\in \Rn{}$,
    such that
    there exists a local diffeomorphism~$f\colon U\rightarrow \tilde{M}$
    preserving the Hessian structures.
\end{proposition}
We prove Proposition~\ref{prop:LocalHessianEmbedding}
in Section~\ref{sec:LocalEmbedding}.
We now prove Theorem~\ref{mthm:LocalEmbedding}
using Proposition~\ref{prop:LocalHessianEmbedding} 
together with the following lemma:
\begin{lemma}\label{lem:CHCAnalyticity}
    A Hessian manifold~$(M,\nabla, g)$ with 
    constant Hessian sectional curvature~$c\in\Rn{}$
    admits an real-analytic structure with respect to which
    the affine connection and the Riemannian metric are analytic.
\end{lemma}
\begin{proof}
    The claim follows from Proposition~\ref{prop:LocalHessianEmbedding},
    since all the standard models naturally
    carry a real-analytic structure with respect to which
    the affine connection and the Riemannian metric are analytic.
\end{proof}
\begin{proof}[Proof of Theorem~\ref{mthm:LocalEmbedding}]
    Let~$M$ be a connected and simply connected Hessian manifold
    with constant Hessian sectional curvature,
    and fix a point~$p\in M$.
    By Proposition~\ref{prop:LocalHessianEmbedding},
    there exists a local diffeomorphism to one of the 
    standard models~$\tilde{M}$ defined on 
    a connected open neighborhood of
    $p$.
    Therefore, by Lemma~\ref{lem:CHCAnalyticity}, 
    $M$ is an analytic Hessian manifold. 
    Moreover, by Proposition~\ref{prop:MaximalityOfModels}, 
    $\tilde{M}$ is a convex maximal analytic Hessian domain. 
    Hence, Proposition~\ref{prop:HessianExtension} 
    implies that~$f$ extends to a local diffeomorphism 
    $f\colon M\rightarrow \tilde{M}$.
\end{proof}

\section{Hessian structure on quotient of \Kahler~manifold}
\label{Sec:QuotientHessian}
In this section, we introduce a Hessian structure on  
quotient manifolds of \Kahler~manifolds 
equipped with suitable Hamiltonian actions 
of abelian Lie groups (Section~\ref{sec:QuotientHessian}),
and investigate some fundamental results,
including the fact 
that the resulting Hessian manifolds
are isomorphic
if and only if the corresponding \Kahler~manifolds are
equivariantly isomorphic under certain conditions
(Section~\ref{sec:Isomorphism}).

\subsection{\Kahler~manifolds}
In this section, we review the definition
of \Kahler~manifolds.
In the following, let~$N$ be a $2n$-dimensional manifold.
We denote by~$TN$ and~$T^*N$ the tangent and cotangent bundles of~$N$,
respectively,
and naturally identify the endomorphism bundle~$\mathrm{End}(TN)$
and the (1,1)-tensor bundle~$TN\otimes T^*N.$

\begin{definition}[\cite{Sil08}]
    (1)~A~$(1,1)$-tensor field~$I$ on~$N$ is called
    an \textit{almost complex structure} if it satisfies
    $I^2=-\id$.\\
    (2)~An almost complex structure~$I$ on~$N$
    is said to be \textit{integrable}
    if, for every point of~$N$,
    there exists a chart around the point with respect to which 
    the matrix representation of~$I$
    is given by~$\begin{pmatrix}
        0&-1_n\\ 1_n &0
    \end{pmatrix}$.
    In this case,~$I$ is called a \textit{complex structure},
    and such a chart is called a \textit{complex chart}.
\end{definition}
\begin{note}
    The complex Euclidean space~$\Cn{n}$
    is identified with~$\Rn{2n}$
    via the map
    $(z^1,\ldots, z^n)\mapsto (x^1,\ldots, x^n, y^1,\ldots, y^n)
    :=(\Re z^1,\ldots, \Re z^n,\Im z^1,\ldots, \Im z^n)$.
    Then, defining an almost complex structure~$I$
    by~$I:=\begin{pmatrix}
        0&-1_n\\ 1_n &0
    \end{pmatrix}$
    with respect to the chart
    $(x^1,\ldots, x^n, y^1,\ldots, y^n)$,
    $I$ is an integrable complex structure on~$\Cn{n}$,
    called the standard complex structure on~$\Cn{n}$.
\end{note}
\begin{definition}
    In this paper, for a complex structure~$I$ on~$N$,
    a smooth map
    $(z^1,\ldots, z^n)\colon
    U\rightarrow \Cn{n}$ defined on an open subset~$U\subset N$
    is said to be a complex chart if the chart
    $(x^1,\ldots, x^n, y^1,\ldots, y^n)
    :=(\Re z^1,\ldots, \Re z^n,\Im z^1,\ldots, \Im z^n)$
    into~$\Rn{2n}$ is a complex chart.
\end{definition}
\begin{fact}\label{fac:Nijenhuis}
    The integrability of an almost complex structure
    can be characterized by the \textit{Nijenhuis tensor}, which is 
    defined for any~$(1,1)$-tensor field~$A$ by
    \begin{equation*}
        N_A(X,Y):=A^2[X,Y]-A([AX,Y]+[X,AY])+[AX,AY].
    \end{equation*}
    An almost complex structure is integrable
    if and only if its Nijenhuis tensor vanishes (the
    Newlander-Nirenberg theorem).
\end{fact}
\begin{definition}[\cite{Sil08}]
    For a complex structure~$I$, a Riemannian metric~$h$,
    and a symplectic structure~$\omega$ on~$N$,
    $(I,h,\omega)$ is called a \textit{\Kahler~structure}
    on~$N$
    if it satisfies the following identity:
    \begin{equation*}
        \omega(\cdot, \cdot)=h(I\cdot, \cdot).
    \end{equation*}
\end{definition}
\begin{lemma}\label{lem:KahlerCompatibility}
    A \Kahler~structure~$(I,h,\omega)$ on~$N$
    satisfies the following identities:
    \begin{equation*}
        h(I\cdot, I\cdot)=h(\cdot, \cdot), 
        \quad
        \omega(I\cdot, I\cdot)=\omega(\cdot, \cdot).
    \end{equation*}
\end{lemma}
\begin{proof}
    The claim follows immediately from the definition.
\end{proof}
\begin{fact}[\cite{Sil08}]\label{fac:KahlerPotential}
    (1)~Let~$\Psi\colon U\rightarrow \Rn{}$
    be a smooth function on
    an open subset~$U\subset \Cn{n}$
    such that the Hermitian matrix
    $\left(
        \frac{\partial^2\Psi}{\partial z^i \partial \bar{z}^j}
    \right)$
    is positive definite at every point.
    Let~$I$ be the standard complex structure,
    and define
    \begin{equation}\label{eq:KahlerPotential}
        \begin{aligned}
            h&:=\frac{\partial^2\Psi}{\partial z^i \partial \bar{z}^j}
            (dz^i\otimes d\bar{z}^j+d\bar{z}^j\otimes dz^i)\\
            \omega&:=\iu
            \frac{\partial^2\Psi}{\partial z^i \partial \bar{z}^j}
            (dz^i\otimes d\bar{z}^j-d\bar{z}^j\otimes dz^i).
        \end{aligned}
    \end{equation}
    Then~$(I,h,\omega)$ is a \Kahler~structure on~$U$.\\
    (2)~Let~$(N, I,h,\omega)$ be a \Kahler~manifold.
    For any~$p\in N$, there exist an open neighborhood
    $U\subset N$, a complex chart~$(z^1,\ldots, z^n)$,
    and a smooth function~$\Psi\colon U\rightarrow \Rn{}$
    such that~$\Psi$ satisfies equation (\ref{eq:KahlerPotential}).
    In this case,~$\Psi$ is called a \textit{\Kahler~potential}.
\end{fact}
\begin{note}\label{not:KahlerPotentialDifference}
    Suppose that smooth functions~$\Psi, \Psi'\colon U\rightarrow \Rn{}$
    on an open subset~$U\subset \Cn{n}$
    satisfy the condition in Fact~\ref{fac:KahlerPotential} (1).
    If~$\Psi'-\Psi=\Re F$ for some holomorphic function
    $F\colon U\rightarrow \Cn{n}$,
    then~$\Psi$ and~$\Psi'$ gives the same 
    \Kahler~structure~$(I,h,\omega)$ on~$U$.
\end{note}
\begin{fact}[{\cite[Theorem 4.3]{KN69}}]\label{fac:KahlerLeviCivita}
    For a \Kahler~manifold~$(N,I,h,\omega)$,
    the Levi-Civita connection~$\nabla^h$ of~$h$
    satisfies~$\nabla^h I=0$.
\end{fact}

\subsection{Hamiltonian action}
In this section, we recall the definition of 
a Hamiltonian action and review its basic properties.
In the following,
let~$N$ be a~$2n$-dimensional manifold,
and let~$\omega$ be a symplectic structure on~$N$.

\begin{definition}[\cite{Sil08}]\label{def:Hamiltonian}
    The action of a Lie group~$G$ on a
    symplectic manifold~$(N,\omega)$ is said to be a
    \textit{Hamiltonian action}
    if it preserves~$\omega$
    and there exists a map~$\mu\colon N\rightarrow
    \mathfrak{g}^*$ satisfying the following conditions:\\
    (i)~For any~$X\in\mathfrak{g}$,~$d\langle \mu, X\rangle=\omega(X^V,\cdot).$\\
    (ii)~For any~$g\in G$,~$Ad_{g}^*:=(Ad_{g^{-1}})^*$ satisfies
    $Ad_{g}^*\circ\mu=\mu\circ g$.\\
    In this case,~$\mu$ is called a \textit{moment map}.
    Here,~$X^V$ denotes the vector field on~$N$ generated by~$X\in\mathfrak{g}$.
\end{definition}
\begin{definition}
    In this paper, 
    a Hamiltonian action on a \Kahler~manifold
    means an action that preserves the \Kahler~structure 
    and is Hamiltonian as an action on the underlying symplectic manifold.
\end{definition}
\begin{lemma}\label{lem:AbelianHamiltonianAction}
    Suppose that a symplectic manifold~$(N,\omega)$ admits
    a Hamiltonian action of an abelian Lie group~$G$.\\
    (1)~For a smooth map~$\mu\colon N\rightarrow \mathfrak{g}^*$,
    condition~(ii) in Definition~\ref{def:Hamiltonian}
    is equivalent to the~$G$-invariance of~$\mu$.\\
    (2)~The subspace of each tangent space of~$N$
    spanned by~$\mathfrak{g}$ is isotropic.
    Namely,
    $\omega(X^V,Y^V)=0$ for any~$X,Y\in\mathfrak{g}$,
    where~$X^V$ denotes the vector field on~$N$
    generated by~$X$.
\end{lemma}
\begin{proof}
    (1) follows from the fact that the adjoint representation
    of~$G$ is trivial, that is,~$Ad_g=\id$ for any~$g\in G$,
    since~$G$ is abelian.
    To prove (2), by (1),~$\mu$ is~$G$-invariant.
    Hence, for any~$X,Y\in\mathfrak{g}$, we have
    $\omega(X^V,Y^V)=d\langle \mu, X\rangle(Y^V)=0$.
\end{proof}

The following holds for an action and a map that preserve the symplectic structure.
\begin{proposition}\label{prop:MomentMaps}
    Let~$\omega, \omega'$ be symplectic structures on
    manifolds~$N, N'$, respectively, 
    and let 
    Lie groups~$G, G'$ act on~$N, N'$
    preserving the symplectic structures, respectively.
    Suppose that
    a local diffeomorphism~$F\colon N'\rightarrow N$
    preserving the symplectic structures 
    and a surjective Lie group 
    homomorphism
    $\rho\colon G'\rightarrow G$ satisfy the following condition:\\
    (i)~For any~$ p\in N'$ and~$g\in G',~ F(g\cdot p)=\rho(g)\cdot F(p).$\\
    Moreover, smooth maps~$\mu\colon N\rightarrow \mathfrak{g}^*$
    and ~$\mu'\colon N'\rightarrow (\mathfrak{g}')^*$
    satisfy the following condition:\\
    (ii)~$\mu'=\rho^*\circ\mu\circ F.$\\
    Then, if~$\mu'$ is a moment map, then~$\mu$ is also a moment map.
\end{proposition}
\begin{proof}
    It follows by a straightforward argument.
\end{proof}

\subsection{Associated \Kahler~manifolds with Hessian manifolds}
\label{sec:TangentBundleKahler}
In this section, we recall the natural construction 
of a \Kahler~manifold from a Hessian manifold.
In particular, the \Kahler~manifolds 
obtained from the standard models (Example~\ref{exm:StandardModels}) 
will be used in Section~\ref{Sec:Correspondence}.

\begin{fact}\label{fac:TangentBundleKahler}
    It is known that the tangent bundle of 
    a Hessian manifold naturally 
    admits a \Kahler~structure (see \cite{Shi07} for details).
    In particular, for a Hessian domain~$(U,\nabla,g=\nabla^2\psi)$, 
    the \Kahler~structure~$(I,h,\omega)$ on~$TU\cong U\times \Rn{n}$ 
    is given by:
    For any~$X, Y\in \Rn{n}$,
    \begin{equation}\label{eq:TangentKahlerStructure}
        \begin{aligned}
        &IX^H:=X^V, \quad
        IX^V:=-X^H\\
        &h(X^H,Y^H)
        =h(X^V,Y^V)
        :=g(X^B, Y^B)\circ\pr_1,\\
        &h(X^H,Y^V)
        =h(X^V,Y^H):=0,\\
        &\omega(X^H,Y^V)
        =-\omega(X^V,Y^H)
        :=g(X^B, Y^B)\circ\pr_1,\\
        &\omega(X^H,Y^H)
        =\omega(X^V,Y^V):=0.
    \end{aligned}
    \end{equation}
    Here,~$\pr_1\colon U\times \Rn{n}\rightarrow U$ 
    and~$\pr_2\colon U\times \Rn{n}\rightarrow \Rn{n}$ denote 
    the projections onto the respective factors.
    For~$X\in\Rn{n}$, 
    regarding~$X$ as a vector field on~$U$, we define~$X^H:=\pr_1^*X$.
    Similarly, regarding~$X$ as a vector field on the vector space~$\Rn{n}$, 
    we define~$X^V:=\pr_2^*X$.
    Then,~$\Psi:=4(\psi\circ\pr_1)$ 
    is a \Kahler~potential for the natural \Kahler~structure.
\end{fact}
In this case,~$U\times\Rn{n}$
naturally carries an effective, proper, free and Hamiltonian
action of~$\Rn{n}$:
\begin{proposition}\label{prop:TangentMomentMap}
    The action of~$\Rn{n}$ 
    on the tangent bundle~$TU\cong U\times \Rn{n}$ 
    by translations in the second factor is an effective, proper, 
    free and Hamiltonian action preserving the \Kahler~structure.
    Moreover, define
    \begin{equation*}
        \nu\colon U\rightarrow (\Rn{n})^*,~
        p\mapsto (X\mapsto -d\psi_p(X)).
    \end{equation*}
    Then~$\mu:=\nu\circ\pr_1\colon U\times \Rn{n}\rightarrow (\Rn{n})^*$
    is a moment map.
    In particular, the~$\Rn{n}$-action generates an isotropic subspace
    at every point on~$U\times \Rn{n}$.
\end{proposition}
\begin{proof}
    The~$\Rn{n}$-action is clearly effective, proper, and free.
    By Fact~\ref{fac:TangentBundleKahler}, 
    the~$\Rn{n}$-action preserves~$X^H$ and~$X^V$ 
    for any~$X\in\Rn{n}$, and hence preserves the \Kahler~structure.
    Next, for~$\mu$, for any~$X,Y\in\Rn{n}$, we have
    \begin{align*}
        d\langle\mu, X\rangle(Y^H)&=d(\langle\nu,X\rangle\circ\pr_1)(Y^H)
        =-d((X\psi)\circ\pr_1)(Y^H)\\
        &=-(XY\psi)\circ\pr_1=-g(X,Y)\circ\pr_1=\omega(X^V, Y^H),\\
        d\langle\mu, X\rangle(Y^V)&=d(\langle\nu,X\rangle\circ\pr_1)(Y^V)=0
        =\omega(X^V, Y^V).
    \end{align*}
    Thus,~$\mu$ is a moment map.
    Finally, by Lemma~\ref{lem:AbelianHamiltonianAction}, 
    the~$\Rn{n}$-action generates an isotropic subspace 
    at every point on~$U\times \Rn{n}$.
\end{proof}
\begin{remark}
    As described in Section~\ref{sec:QuotientHessian}, 
    for a Hessian domain~$U$, 
    the quotient manifold~$(U\times \Rn{n})/\Rn{n}$ 
    naturally admits a Hessian structure and 
    is isomorphic to~$U$ as a Hessian manifold 
    (Lemma~\ref{lem:TangentBundleQuotient}).
\end{remark}

The following holds for the curvature.
\begin{fact}[\cite{Shi07}]\label{fac:TangentCurvature}
    For a Hessian manifold~$M$ ans the tangent bundle~$TM$,
    the following conditions are equivalent:\\
    (a)~$M$ has constant Hessian sectional curvature~$c$.\\
    (b)~$TM$ has constant holomorphic sectional curvature~$-c$.
\end{fact}
See \cite{KN69} for the definition of holomorphic sectional curvature.

We then compute 
the \Kahler~potentials and moment maps 
determined by the Hessian potentials 
for the standard models (Example~\ref{exm:StandardModels}).
\begin{note}\label{not:TangentHolomorphicChart}
    For a Hessian domain~$(U,\nabla,g)$,
    $U\times \Rn{n}$ is naturally identified with~$U+\iu\Rn{n}\subset \Cn{n}$
    as complex manifolds via the map
    $U\times \Rn{n}\rightarrow U+\iu\Rn{n}\subset \Cn{n},~ (p,v)\mapsto p+\iu v$.
    In particular,
    for the standard chart~$(\theta^1,\ldots,\theta^n)$ on~$U$
    and the standard chart~$(v^1,\ldots, v^n)$ on~$\Rn{n}$,
    the functions~$(w^1,\ldots, w^n)$ defined by 
    $w^i:=\theta^i\circ\pr_1 +\iu v^i$
    form a global complex chart on~$U\times \Rn{n}$.
\end{note}
\begin{example}\label{exm:StandardModelsKahler}
    On the tangent bundles of the standard models 
    (Example~\ref{exm:StandardModels}), 
    the \Kahler~potential~$\Psi$ 
    defined in Fact~\ref{fac:TangentBundleKahler} 
    and the moment map~$\mu$ defined 
    in Proposition~\ref{prop:TangentMomentMap} 
    are given in the following form 
    with respect to the global complex chart
    $(w^1,\ldots,w^n)$ in Note~\ref{not:TangentHolomorphicChart}.

    (1)~Let~$c<0$. For~$\mathcal{S}^n(c)$,
    \begin{align*}
        \Psi&=-\frac{4}{c}\log\left(1+e^{\Re w^1}+\cdots+e^{\Re w^n}\right),\\
        \mu&=
        \frac{1}{c\left(1+e^{\Re w^1}+\cdots+e^{\Re w^n}\right)}
        \left(e^{\Re w^1}, \ldots, e^{\Re w^n}\right).
    \end{align*}
    (2)~Let~$k=0,\ldots, n$. For~$M:=\mathcal{R}^{n,k}(0)$,
    \begin{align*}
        \Psi&=2(\Re w^1)^2+\cdots+2(\Re w^k)^2+4e^{\Re w^{k+1}}
            +\cdots+4e^{\Re w^{n}},\\
        \mu&=-\left(\Re w^1,\ldots, \Re w^k,
        e^{\Re w^{k+1}},\ldots, e^{\Re w^n}\right).
    \end{align*}
    (3)~Let~$c>0$.\\
    (3-i)~For~$\mathcal{H}^n(c)$,
    \begin{align*}
        \Psi&=-\frac{4}{c}\log\left(1-e^{\Re w^1}-\cdots-e^{\Re w^n}\right),\\
        \mu&=\frac{1}{c\left(1-e^{\Re w^1}+\cdots-e^{\Re w^n}\right)}
        \left(e^{\Re w^1}, \ldots, e^{\Re w^n}\right).
    \end{align*}
    (3-ii)~For~$\mathcal{H}^{n,\mathrm{H}}(c)$,
    \begin{align*}
        \Psi&=-\frac{4}{c}\log\left(\cos {\Re w^1}-e^{\Re w^2}-\cdots-e^{\Re w^n}\right),\\
        \mu&=\frac{1}{c\left(\cos\Re w^1-e^{\Re w^2}-\cdots-e^{\Re w^n}\right)}
    \left(\sin \Re w^1, e^{\Re w^2},\ldots,e^{\Re w^2}\right).
    \end{align*}
    (3-iii)~Let~$k=1,\ldots, n$. For~$\mathcal{H}^{n,k}(c)$,
    \begin{align*}
        \Psi&=-\frac{4}{c}\log\left(-\Re w^1-\frac{1}{2}(\Re w^2)^2-\cdots
        -\frac{1}{2}(\Re w^k)^2
            -e^{\Re w^{k+1}}-\cdots-e^{\Re w^n}\right),\\
        \mu&=\frac{\left(1, \Re w^2,\ldots, \Re w^k,
    e^{\Re w^{k+1}},\ldots, e^{\Re w^n}\right)}{c\left(-\Re w^1-
    \frac{1}{2}(\Re w^2)^2-\cdots- \frac{1}{2}(\Re w^k)^2
    -e^{\Re w^{k+1}}-\cdots-e^{\Re w^n}\right)}.
    \end{align*}
\end{example}

\subsection{Hessian structures on quotient of \Kahler~manifolds}
\label{sec:QuotientHessian}
In this section, we introduce a Hessian structure on  
quotient manifolds of \Kahler~manifolds 
equipped with suitable Hamiltonian actions 
of abelian Lie groups.

\begin{setting}\label{set:KahlerQuotient}
    Let~$(N,I,h,\omega)$ be a connected~$2n$-dimensional
    \Kahler~manifold equipped with an effective proper
    Hamiltonian action of a connected~$n$-dimensional
    abelian Lie group~$\mathbb{A}$.
\end{setting}
\begin{definition}
    Consider Setting~\ref{set:KahlerQuotient}.
    For~$X\in \mathfrak{a}$,
    we denote by~$X^V$ the vector field on~$N$ generated by~$X$,
    and define~$X^H:=-IX^V$.
    Moreover, we define
    \begin{equation*}
        N^\circ_\mathbb{A}:=\{p\in N~:~ \mathbb{A}_p=\{e\}\},
    \end{equation*}
    where~$\mathbb{A}_p$ denotes the isotropy subgroup of~$\mathbb{A}$
    at~$p\in N$.
\end{definition}
\begin{note}
    Since~$I$ is~$\mathbb{A}$-invariant,
    $X^H$ is also~$\mathbb{A}$-invariant.
\end{note}
\begin{lemma}\label{lem:LieBrackets}
    Consider Setting~\ref{set:KahlerQuotient}.
    For any~$X, Y\in\mathfrak{a}$, we have
    \begin{equation*}
        [X^H, Y^H]=[X^V, Y^H]=[X^V, Y^V]=0.
    \end{equation*}
\end{lemma}
\begin{proof}
    Since~$\mathbb{A}$ is abelian, we have~$[X^V,Y^V]=0.$
    Since~$I$ is~$\mathbb{A}$-invariant,
    for any vector field~$A$ on~$N$,
    \begin{equation*}
        0=(\mathcal{L}_{X^V}I)A=[X^V, IA]-I[X^V, A].
    \end{equation*}
    In particular,
    \begin{equation*}
        [X^V, Y^H]=[X^V, -IY^V]=-I[X^V,Y^V]=0.
    \end{equation*}
    Moreover, by Fact~\ref{fac:Nijenhuis},
    the Nijenhuis tensor~$N_I$ of~$I$ vanishes.
    Thus,
    \begin{equation*}
        0=N_I(X^V, Y^V)
        =-[X^V,Y^V]-I([IX^V, Y^V]+[X^V, IY^V])+[IX^V, IY^V]
        =[X^H,Y^H].
    \end{equation*}
\end{proof}
\begin{lemma}\label{lem:Lagrangeness}
    Consider Setting~\ref{set:KahlerQuotient}.\\
    (1)~For any~$X\in \mathfrak{a}$, we have
    \begin{equation*}
        IX^H= X^V,\quad IX^V=-X^H.
    \end{equation*}
    (2)~For any~$X, Y\in \mathfrak{a}$, we have
    \begin{alignat*}{2}
        &h(X^H,Y^H)=h(X^V,Y^V), \quad&
        &h(X^H,Y^V)=h(X^V,Y^H)=0,\\
        &\omega(X^H,Y^V)=-\omega(X^V,Y^H), \quad&
        &\omega(X^H,Y^H)=\omega(X^V,Y^V)=0.
    \end{alignat*}
    (3)~For any~$X,Y ,Z\in\mathfrak{a}$, we have
    \begin{equation*}
        X^Hh(Y^H,Z^H)=Y^Hh(X^H,Z^H).
    \end{equation*}
\end{lemma}
\begin{proof}
    (1) follows immediately from the definition.
    For~(2), 
    since the moment map~$\mu$ is~$\mathbb{A}$-invariant,
    for any~$X,Y\in\mathfrak{a}$,
    $0=d\langle\mu, X\rangle(Y^V)=\omega(X^V,Y^V)$.
    The rest follows from Lemma~\ref{lem:KahlerCompatibility}.

    For (3), since~$d\omega=0$, for any~$X,Y,Z\in\mathfrak{a}$,
    by Lemma~\ref{lem:LieBrackets},
    \begin{align*}
        0&=(d\omega)(X^H,Y^H,Z^V)\\
        &=X^H\omega(Y^H,Z^V)-Y^H\omega(X^H, Z^V)+Z^V\omega(X^H, Y^H)\\
        &=X^Hh(Y^H, Z^H)-Y^Hh(X^H, Z^H)+0,
    \end{align*}
    where the fact that~$Z^V\omega(X^H, Y^H)=0$ follows from
    the~$\mathbb{A}$-invariance.
\end{proof}

We consider the quotient manifolds.
\begin{definition}
    Consider Setting~\ref{set:KahlerQuotient}.
    For the quotient manifold~$N^\circ_\mathbb{A}/\mathbb{A}$,
    we denote the quotient map 
    by~$\pi\colon N^\circ_\mathbb{A}\rightarrow N^\circ_\mathbb{A}/\mathbb{A}$.
    For each~$X\in \mathfrak{a}$, 
    we also define
    $X^B:=\pi_*X^H\in \mathfrak{X}(N^\circ_{\mathbb{A}}/\mathbb{A})$.
\end{definition}
\begin{lemma}\label{lem:BaseLieBracket}
    For any~$X,Y\in\mathfrak{a}$, we have~$[X^B, Y^B]=0$.
\end{lemma}
\begin{proof}
    This follows immediately from Lemma~\ref{lem:LieBrackets}.
\end{proof}
\begin{fact}[\cite{GGK02}]\label{fac:OpenDense}
    Suppose that a connected abelian Lie group~$G$
    acts on a connected manifold~$N$ effectively and properly.
    Then,~$N^\circ:=\{p\in N~:~G_p=\{e\}\}$
    is an open and dense subset of~$N$.
    If~$N$ is orientable, then~$N^\circ$ is connected.
\end{fact}
\begin{lemma}\label{lem:OpenDense}
    Consider Setting~\ref{set:KahlerQuotient}.
    Then~$N^\circ_\mathbb{A}$ is open, dense, and connected,
    and the quotient manifold~$N^\circ_{\mathbb{A}}/\mathbb{A}$
    is connected.
\end{lemma}
\begin{proof}
    Since~$N$ is, in particular, a complex manifold,
    it is orientable. Hence the assertion follows from
    Fact~\ref{fac:OpenDense}.
\end{proof}

We introduce a Hessian structure on the quotient manifold.
\begin{lemma}\label{lem:FlatConnection}
    Let~$V\subset \mathfrak{X}(M)$ be 
    an~$n$-dimensional vector space
    of vector fields
    on an~$n$-dimensional manifold~$M$.
    If~$V$ satisfies the following conditions,
    then there exists a unique flat affine connection~$\nabla$
    satisfying~$V=\mathfrak{X}_\nabla(M)$.\\
    (i)~For any~$p\in M$,~$\{X_p\in T_pM~:~ X\in V\}=T_pM.$\\
    (ii)~For any~$X,Y\in V$,~$[X,Y]=0.$
\end{lemma}
\begin{proof}
    The affine connection~$\nabla$ on~$M$
    is uniquely determined by requiring
    $\nabla X=0$ for any~$X\in V$.
    Then, by a direct computation,~$\nabla$ is flat.
\end{proof}
\begin{fact}[\cite{Lee18}]
    \label{fac:RiemannianQuotient}
    If a Lie group~$G$ acts freely and properly on a Riemannian manifold~$N$
    by isometries,
    then~$N/G$ admits a unique Riemannian metric making
    $\pi\colon N\rightarrow N/G$ a Riemannian submersion.
    Here, a submersion~$f\colon N\rightarrow N'$ between Riemannian manifolds~$N$
    and~$N'$ is called a Riemannian submersion if, for every point~$p\in N$,
    the differential~
    \begin{equation*}
        (f_*)_p\colon (\Ker (f_*)_p)^{\perp}\rightarrow T_{f(p)}N'
    \end{equation*}
    restricted to the subspace~$(\Ker (f_*)_p)^{\perp}\subset T_pN$
    is an isometry.
\end{fact}

\begin{lemma}\label{lem:HorizontalDistribution}
    For any~$p\in N^\circ_\mathbb{A}$,
    we have
    \begin{align*}
        \ker{(\pi_*)_p}&=\{X^V_p\in T_pN~:~ X\in\mathfrak{a}\}\\
        (\ker{(\pi_*)_p})^\perp&=\{X^H_p\in T_pN~:~ X\in\mathfrak{a}\}
        =I\ker{(\pi_*)_p}.
    \end{align*}
\end{lemma}
\begin{proof}
    The first identity follows from
    $\ker d\pi_p=T_p(\mathbb{A}\cdot p)$.
    Regarding the second identity, 
    the right-hand equality is clear by definition. 
    To prove the left-hand equality, 
    Lemma~\ref{lem:Lagrangeness} first yields 
    the inclusion~$\supset$.
    On the other hand, by the first identity, the map
    $\mathfrak{a}\rightarrow \ker{(\pi_*)_p},~X\mapsto X^V_p$
    is a linear isomorphism. Hence, the linear subspace
    $\{X^H_p\in T_pN~:~ X\in\mathfrak{a}\}$ has dimension~$n$.
    Since~$(\ker{(\pi_*)_p})^\perp$ also has dimension~$n$,
    the reverse inclusion~$\subset$ follows.
\end{proof}

\begin{propdef}\label{propdef:QuotientHessian}
    There exist a unique Hessian structure~$(\nabla, g)$
    on
    $N^\circ_\mathbb{A}/\mathbb{A}$ 
    satisfying the following conditions:\\
    (i)~$\pi$ is a Riemannian submersion.\\
    (ii)~For any~$X\in\mathfrak{a}$,~$\nabla X^B=0.$

    In this case, for any~$X, Y\in \mathfrak{a}$,
    we have
    \begin{equation*}
        h(X^H,Y^H)=h(X^V,Y^V)=g(X^B,Y^B)\circ \pi.
    \end{equation*}
    Throughout the following, we equip 
    $N^\circ_{\mathbb{A}}/\mathbb{A}$ with 
    the Hessian structure defined above.
\end{propdef}
\begin{proof}
    A flat affine connection~$\nabla$ and
    a Riemannian metric~$g$ on~$N_\mathbb{A}^\circ/\mathbb{A}$
    are uniquely determined  by 
    Lemma~\ref{lem:FlatConnection}
    and
    Fact~\ref{fac:RiemannianQuotient}, respectively.
    The identity then follows from
    Lemma~\ref{lem:HorizontalDistribution}.
    Moreover, it follows from Fact~\ref{fac:HessianCondition} and 
    Lemma~\ref{lem:Lagrangeness}~(3)
    that~$(\nabla, g)$ defines a Hessian structure
    on~$N^\circ_{\mathbb{A}}/\mathbb{A}$.
\end{proof}
\begin{note}\label{not:QuotientParallelVector}
    Consider Setting~\ref{set:KahlerQuotient}.
    For the Hessian structure~$(\nabla, g)$ on the quotient manifold
    $N^\circ_{\mathbb{A}}/\mathbb{A}$ 
    defined in Proposition-Definition~\ref{propdef:QuotientHessian},
    $\mathfrak{X}_\nabla(N^\circ_{\mathbb{A}}/\mathbb{A})$ 
    is an abelian Lie algebra, and the map
    $\mathfrak{a}\rightarrow \mathfrak{X}_\nabla(N^\circ_{\mathbb{A}}/\mathbb{A}),~
    X\mapsto X^B$
    is an isomorphism of Lie algebras.
\end{note}

The following result holds 
for the tangent bundle of a Hessian domain 
(Section~\ref{sec:TangentBundleKahler}).
\begin{lemma}\label{lem:TangentBundleQuotient}
    For a Hessian domain~$U$, 
    suppose that the tangent bundle
    $TU\cong U\times \Rn{n}$
    carries the \Kahler~structure 
    and the effective, proper, free, and Hamiltonian~$\Rn{n}$-action
    defined in Section~\ref{sec:TangentBundleKahler}.
    Then
    the quotient manifold~$(U\times \Rn{n})/\Rn{n}$ 
    is isomorphic to~$U$ as a Hessian manifold.
\end{lemma}
\begin{proof}
    By definition,
    the projection~$\pr_1\colon U\times \Rn{n}\rightarrow U$
    induces a Hessian isomorphism
    $(U\times \Rn{n})/\Rn{n}\rightarrow U$
    via the 
    quotient map~$\pi\colon U\times \Rn{n}\rightarrow (U\times \Rn{n})/\Rn{n}$.
\end{proof}

\subsection{Isomorphism}\label{sec:Isomorphism}
In this section, we discuss 
Hessian manifolds obtained as quotients 
of \Kahler~manifolds by suitable group actions, 
together with structure-preserving maps between them.
We show that suitable maps between \Kahler~manifolds 
that are compatible with the group actions 
induce maps preserving the Hessian structures 
on the quotient manifolds (Proposition~\ref{prop:QuotientIsomorphism}), 
and that the converse holds for \Kahler~manifolds 
satisfying certain conditions 
(Proposition~\ref{prop:HessianKahlerIsomorphism}).
In particular, we establish an equivalence 
between isomorphisms of \Kahler~manifolds, including their group actions, 
and isomorphisms of the quotient manifolds as Hessian manifolds 
(Corollary~\ref{cor:IsomorphismEquivalence}).
We also prove Proposition~\ref{prop:QuotientMorphism}, 
which will be used in Section~\ref{Sec:Correspondence}.

\begin{setting}\label{set:Kahlers}
    Let~$(N,I,h,\omega), (N',I',h',\omega')$ be connected
    $2n$-dimensional \Kahler~manifolds, each equipped with an
    effective proper Hamiltonian action of a connected
    $n$-dimensional abelian Lie group~$\mathbb{A}, \mathbb{A}'$,
    respectively.
    
    We also set
    $N^\circ_{\mathbb{A}}:=\{p\in N~:~\mathbb{A}_p=\{e\}\}$,
    and assume that the quotient manifold
    $N^\circ_{\mathbb{A}}/\mathbb{A}$ is equipped with the
    Hessian structure~$(\nabla, g)$ defined in
    Proposition-Definition~\ref{propdef:QuotientHessian}.
    We denote the quotient map by
    $\pi\colon N^\circ_{\mathbb{A}}\rightarrow N^\circ_{\mathbb{A}}/\mathbb{A}$.
    We define
    $(N')^\circ_{\mathbb{A}'}, (\nabla', g')$
    and the quotient map
    $\pi'\colon (N')^\circ_{\mathbb{A}'}\rightarrow (N')^\circ_{\mathbb{A}'}/\mathbb{A}'$
    similarly.

    Furthermore, for an arbitrary~$X\in \mathfrak{a}:=\mathrm{Lie}(\mathbb{A})$,
    let~$X^V$ denote the vector field on~$N$ generated by~$X$,
    and set~$X^H:=-IX^V$.
    We also set
    $X^B:=\pi_*X^H \in\mathfrak{X}_{\nabla}(N^\circ_{\mathbb{A}}/\mathbb{A})$.
\end{setting}

As a preliminary step, 
we investigate the \Kahler~manifolds~$N, N'$ 
and their quotient manifolds
$N^\circ_{\mathbb{A}}/\mathbb{A}, (N')^\circ_{\mathbb{A}'}/\mathbb{A}'$, 
as well as maps between them.
\begin{lemma}\label{lem:Isomorphism}
    Consider Setting~\ref{set:Kahlers}.\\
    (1)~Let~$F\colon N\rightarrow N'$ be a smooth map satisfying
    $F(N)\cap (N')^\circ_{\mathbb{A}'}\neq \emptyset$.
    Then the following conditions are equivalent.\\
    (1-a)~There exists a homomorphism~$\rho\colon \mathbb{A}\rightarrow \mathbb{A}'$
    such that~$F\circ a=\rho(a)\circ F$ for any~$a\in\mathbb{A}$.\\
    (1-b)~There exists a homomorphism
    $R\colon \mathfrak{a}\rightarrow \mathfrak{a}'
    :=\mathrm{Lie}(\mathbb{A}')$
    such that
    $F_*X^V_p=(RX)^V_{F(p)}$ for any
    $p\in N$
    and any 
    $X\in\mathfrak{a}$.\\
    In this case, 
    $R=\rho_*$ holds.\\
    (2)~Let~$f\colon N^\circ_{\mathbb{A}}/\mathbb{A}
    \rightarrow (N')^\circ_{\mathbb{A}'}/\mathbb{A}'$ be a smooth map.
    Then the following conditions are equivalent.\\
    (2-a)~$f$ is a connection-preserving local diffeomorphism.\\
    (2-b)~There exists a Lie algebra isomorphism
    $R\colon \mathfrak{a}\simarrow \mathfrak{a}'$
    such that 
    $f_*X^B_p=(RX)^B_{f(p)}$ for any
    $p\in N^\circ_{\mathbb{A}}/\mathbb{A}$ and any 
    $X\in\mathfrak{a}$.
\end{lemma}
\begin{proof}
    The implication (1-a)$\Rightarrow$(1-b) 
    follows by setting 
    $R=\rho_*$.

    We show (1-b)$\Rightarrow$(1-a).
    First, we show that~$R$ induces a Lie group homomorphism
    $\rho\colon \mathbb{A}\rightarrow\mathbb{A}'$.
    Since~$\mathbb{A}$ is abelian, it suffices to show that
    \begin{equation}\label{eq:exponential}
        \forall X\in\mathfrak{a},~(\exp X=e\in \mathbb{A}
        \Rightarrow \exp (RX)=e\in\mathbb{A}').
    \end{equation}
    Let~$p\in N$ be a point such that
    $F(p)\in (N')^\circ_{\mathbb{A}'}$.
    Then, for any~$X\in\mathfrak{a}$, the uniqueness of integral curves
    yields
    \begin{equation}\label{eq:Equivariance}
        F(\exp X\cdot p)=\exp (RX)\cdot F(p).
    \end{equation}
    Hence, \eqref{eq:exponential} holds, and~$R$ induces a Lie group
    homomorphism 
    $\rho\colon \mathbb{A}\rightarrow\mathbb{A}'$.
    Since~$\mathbb{A}$ is connected and abelian, the exponential map
    $\exp\colon\mathfrak{a}\rightarrow\mathbb{A}$ is surjective.
    Thus, \eqref{eq:Equivariance} again implies that
    $F\circ a=\rho(a)\circ F$ for any 
    $a\in\mathbb{A}$.

    We show (2-a)$\Rightarrow$(2-b).
    By Lemma~\ref{lem:OpenDense}, the manifolds
    $M:=N^\circ_{\mathbb{A}}/\mathbb{A},~
    M':=(N')^\circ_{\mathbb{A}'}/\mathbb{A}'$
    are connected.
    Hence, by Note~\ref{not:AffineMapParallelVector}, the linear map
    $f_*\colon \mathfrak{X}_{\nabla}(M)\rightarrow
        \mathfrak{X}_{\nabla'}(M')$
    induced by~$f$ is a Lie algebra homomorphism 
    between the abelian
    Lie algebras 
    $\mathfrak{X}_{\nabla}(M)$ and
    $\mathfrak{X}_{\nabla'}(M')$.
    Moreover, since~$f$ is a local diffeomorphism, 
    $f_*$ is an
    isomorphism.
    Therefore, by Note~\ref{not:QuotientParallelVector}, we have
    Lie algebra isomorphisms
    $\mathfrak{a}\cong \mathfrak{X}_{\nabla}(M),~
        \mathfrak{a}'\cong \mathfrak{X}_{\nabla'}(M').$
    Consequently, there exists a Lie algebra isomorphism
    $R\colon \mathfrak{a}\simarrow\mathfrak{a}'$
    satisfying the required condition.

    We show (2-b)$\Rightarrow$(2-a).
    By Note~\ref{not:QuotientParallelVector}, the Lie algebra
    isomorphism induced by~$R$,
    $\mathfrak{X}_{\nabla}(M)\simarrow
        \mathfrak{X}_{\nabla'}(M'),\quad
        X^B\mapsto (RX)^B,$
    coincides with~$f_*$ by the condition
    $f_*X^B_p=(RX)^B_{f(p)}$.
    Thus,~$f$ maps every parallel vector field on~$M$ to a parallel
    vector field on~$M'$.
    It follows from Note~\ref{not:AffineMapParallelVector} that
    $f$ preserves the connections.
\end{proof}
\begin{lemma}\label{lem:KahlerHessianMorphism}
    Consider Setting~\ref{set:Kahlers}.
    Let~$F\colon N\rightarrow N'$ be a smooth map satisfying
    $F(N^\circ_{\mathbb{A}})\subset (N')^\circ_{\mathbb{A}'}$,
    and let 
    $f\colon N^\circ_{\mathbb{A}}/\mathbb{A}
    \rightarrow (N')^\circ_{\mathbb{A}'}/\mathbb{A}'$ be a smooth map
    satisfying 
    $f\circ\pi=\pi'\circ F$.
    In other words, the following diagram commutes.
    \center
    \begin{tikzpicture}[auto]
        \node (x) at (-1, 0) {$N^\circ_{\mathbb{A}}/\mathbb{A}$};
        \node (a) at (-1, 2) {$N^\circ_{\mathbb{A}}$};
        \node (y) at (1, 0) {$(N')^\circ_{\mathbb{A}'}/\mathbb{A}'$};
        \node (b) at (1, 2) {$(N')^\circ_{\mathbb{A}'}$};
        \draw[->] (a) to node {$F$} (b);
        \draw[->] (x) to node {$f$} (y);
        \draw[->] (a) to node {$\pi$} (x);
        \draw[->] (b) to node {$\pi'$} (y);
    \end{tikzpicture}
    \flushleft
    (1)~For a Lie algebra homomorphism
    $R\colon \mathfrak{a}\rightarrow
    \mathfrak{a}'$, the following conditions are equivalent:\\
    (1-a)~For any~$p\in N^\circ_{\mathbb{A}}$ and any 
    $X\in \mathfrak{a}$,
    $F_*X^V_p=(RX)^V_{F(p)}$ and 
    $F_*X^H_p=(RX)^H_{F(p)}.$\\
    (1-b)~$F_*I=I'F_*$ and, for any 
    $p\in N^\circ_{\mathbb{A}}/\mathbb{A}$
    and any 
    $X\in \mathfrak{a}$,
    $f_*X^B_p=(RX)^B_{f(p)}.$\\
    (2)~If~$F_*I=I'F_*$, then
    $F^*h'=h~\Leftrightarrow~f^*g'=g.$
\end{lemma}
\begin{proof}
    We first discuss the quotient maps~$\pi, \pi'$.
    By Lemma~\ref{lem:HorizontalDistribution}, for any
    $p\in N^\circ_{\mathbb{A}}$, we have
    $(\ker(\pi_*)_p)^\perp=I\ker(\pi_*)_p$, and
    \begin{equation}\label{eq:OrthoDecomp}
        T_pN=\ker(\pi_*)_p\oplus I\ker(\pi_*)_p
    \end{equation}
    gives an orthogonal direct sum decomposition.
    Hence, the restriction
    $\pi_*|_{I\ker(\pi_*)_p}\colon I\ker(\pi_*)_p
    \simarrow T_{\pi(p)}(N^\circ_{\mathbb{A}}/\mathbb{A})$
    is a linear isomorphism, which is an isometry with respect to
    $h$ and 
    $g$ by Proposition-Definition~\ref{propdef:QuotientHessian}.
    Here, we also note that, by Lemma~\ref{lem:HorizontalDistribution},
    \begin{equation*}
        \ker(\pi_*)_p=\{X^V_p ~:~ X\in\mathfrak{a}\},\quad
        I\ker(\pi_*)_p=\{X^H_p ~:~ X\in\mathfrak{a}\}.
    \end{equation*}

    Regarding the maps~$F$ and~$f$, the condition
    $f\circ\pi=\pi'\circ F$ implies that
    $F_*\ker(\pi_*)_p\subset \ker(\pi'_*)_{F(p)}$
    for any 
    $p\in N^\circ_{\mathbb{A}}$.
    Moreover, if
    $F_*I=I'F_*$ holds, then
    $F_*I\ker(\pi_*)_p\subset I'\ker(\pi'_*)_{F(p)}$.
    Thus, again by the condition 
    $f\circ\pi=\pi'\circ F$, we obtain
    \begin{equation}\label{eq:DifferentiateCommute}
        f_*\circ \pi_*|_{I\ker(\pi_*)_p}=
        (\pi'_*|_{I'\ker(\pi'_*)_{F(p)}})\circ (F_*|_{I\ker(\pi_*)_p})
    \end{equation}
    for any 
    $p\in N^\circ_{\mathbb{A}}$.

    We show (1-a)$\Rightarrow$(1-b).
    First, we show that~$F_*I=I'F_*$ holds.
    By Lemma~\ref{lem:OpenDense}, it suffices to verify this on
    $N_\mathbb{A}^\circ$.
    For any 
    $p\in N^\circ_{\mathbb{A}}$ and any 
    $X\in\mathfrak{a}$,
    we have
    \begin{align*}
        F_*IX^V_p&=-F_*X^H_p=-(RX)^H_{F(p)}=I'(RX)^V_{F(p)}
        =I'F_*X^V_p,\\
        F_*IX^H_p&=F_*X^V_p=(RX)^V_{F(p)}=I'(RX)^H_{F(p)}
        =I'F_*X^H_p.
    \end{align*}
    This proves the assertion.
    Hence, \eqref{eq:DifferentiateCommute} holds, and for any
    $p\in N^\circ_{\mathbb{A}}$ and any 
    $X\in\mathfrak{a}$, we obtain
    \begin{equation*}
        f_*X^B_{\pi(p)}
        =f_*\pi_*X^H_{p}
        =\pi'_*F_*X^H_p
        =\pi'_*(RX)^H_{F(p)}
        =(RX)^B_{f(\pi(p))}.
    \end{equation*}

    We show (1-b)$\Rightarrow$(1-a).
    By the condition 
    $F_*I=I'F_*$, equation~\eqref{eq:DifferentiateCommute}
    holds. For any 
    $p\in N^\circ_{\mathbb{A}}$ and any
    $X\in\mathfrak{a}$, we have
    \begin{equation*}
        \pi'_*F_*X^H_p
        =f_*\pi_*X^H_{p}
        =f_*X^B_{\pi(p)}
        =(RX)^B_{f(\pi(p))}
        =\pi'_*(RX)^H_{F(p)}.
    \end{equation*}
    Therefore, 
    $F_*X^H_p=(RX)^H_{F(p)}$.
    Again by the condition 
    $F_*I=I'F_*$, we obtain
    \begin{equation*}
        F_*X^V_p=F_*IX^H_p=I'F_*X^H_p=I'(RX)^H_{F(p)}
        =(RX)^V_{F(p)}.
    \end{equation*}

    For (2), equation~\eqref{eq:DifferentiateCommute} holds by the
    condition 
    $F_*I=I'F_*$.
    Moreover, by the orthogonal direct sum decomposition
    \eqref{eq:OrthoDecomp}, for any 
    $p\in N^\circ_\mathbb{A}$,
    the restriction
    $F_*|_{I\ker(\pi_*)_p}\colon
    I\ker(\pi_*)_p\rightarrow I\ker(\pi'_*)_{F(p)}$
    is an isometry if and only if
    $f_*\colon T_{\pi(p)}(N^\circ_{\mathbb{A}}/\mathbb{A})
    \rightarrow T_{f(\pi(p))}((N')^\circ_{\mathbb{A}'}/\mathbb{A}')$
    is an isometry.
    This proves 
    $(\Rightarrow)$.

    We show 
    $(\Leftarrow)$.
    By Lemma~\ref{lem:OpenDense}, it suffices to verify the assertion
    on 
    $N_\mathbb{A}^\circ$.
    By assumption, the restriction
    $F_*|_{I\ker(\pi_*)_p}\colon
    I\ker(\pi_*)_p\rightarrow I\ker(\pi'_*)_{F(p)}$
    is an isometry.
    Moreover, we have
    \begin{align*}
        (F^*h')(I\cdot, I\cdot)
        &=h'(F_*I\cdot, F_*I\cdot)\circ F\\
        &=h'(I'F_*\cdot, I'F_*\cdot)\circ F\\
        &=h'(F_*\cdot, F_*\cdot)\circ F
        &&(\textrm{Note}~\ref{lem:KahlerCompatibility})\\
        &=(F^*h')(\cdot, \cdot),\\
        h(I\cdot, I\cdot)&=h(\cdot,\cdot)
        &&(\textrm{Lemma}~\ref{lem:KahlerCompatibility}).
    \end{align*}
    It follows that the restriction
    $F_*|_{\ker(\pi_*)_p}\colon
    \ker(\pi_*)_p\rightarrow \ker(\pi'_*)_{F(p)}$
    is also an isometry.
    Consequently, 
    $F^*h'=h$.
\end{proof}

We define an isomorphism between
\Kahler~manifolds compatible with the group actions as follows.
\begin{definition}\label{def:EquivariantIsomorphism}
    In Setting~\ref{set:Kahlers}, let 
    $F:N\simarrow N'$ be an
    isomorphism of \Kahler~manifolds and let
    $\rho\colon\mathbb{A}\simarrow \mathbb{A}'$ be a group isomorphism.
    We say that 
    $F$ is an equivariant isomorphism 
    with respect to $\rho$
    if
    $F\circ a=\rho(a)\circ F$ holds for any 
    $a\in\mathbb{A}$.
\end{definition}
\begin{note}\label{not:EquivariantIsomorphism}
    Consider Definition~\ref{def:EquivariantIsomorphism}.\\
    (1)~For a fixed~$F$, 
    the group isomorphism~$\rho$ is unique.\\
    (2)~For any~$a\in\mathbb{A}$, the map~$a\colon N\simarrow N$
    is an equivariant isomorphism with respect to
    $\rho:=\id\colon \mathbb{A}\simarrow \mathbb{A}$.\\
    (3)~For any equivariant isomorphism
    $F\colon N\simarrow N'$ with respect to
    $\rho\colon \mathbb{A}\simarrow \mathbb{A}'$
    and any 
    $a'\in \mathbb{A}'$, the map
    $a'\circ F\colon N\simarrow N'$ is an equivariant isomorphism
    with respect to 
    $\rho$.
\end{note}
In this case, 
an isomorphism of the Hessian structures 
on the quotient manifolds is naturally induced.
\begin{proposition}\label{prop:QuotientIsomorphism}
    In the setting of Setting~\ref{set:Kahlers},
    suppose that 
    $F:N\simarrow N'$ is an equivariant isomorphism
    with respect to 
    $\rho\colon \mathbb{A}\simarrow \mathbb{A}'$.
    Then there exists a unique map
    $f\colon N^\circ_{\mathbb{A}}/\mathbb{A}\simarrow
    (N')^\circ_{\mathbb{A}'}/\mathbb{A}'$
    satisfying 
    $f\circ \pi =\pi'\circ F$,
    and 
    $f$ is an isomorphism of Hessian manifolds.
\end{proposition}
\begin{proof}
    First, for any 
    $p\in N^\circ_{\mathbb{A}}$, the isotropy subgroups
    satisfy 
    $\rho(\mathbb{A}_p)=\mathbb{A}'_{F(p)}$.
    Hence, 
    $F(N^\circ_{\mathbb{A}})\subset (N')^\circ_{\mathbb{A}'}$.
    Therefore, by the condition 
    $F\circ a=\rho(a)\circ F$,
    $F$ maps 
    $\mathbb{A}$-orbits to 
    $\mathbb{A}'$-orbits.
    It follows that there exists a unique map
    $f\colon N^\circ_{\mathbb{A}}/\mathbb{A}\rightarrow
    (N')^\circ_{\mathbb{A}'}/\mathbb{A}'$
    satisfying 
    $f\circ\pi =\pi'\circ F$.
    Thus, it remains to show that~$f$ preserves the Hessian structure,
    which reduces the problem to the following
    Proposition~\ref{prop:QuotientMorphism}.
\end{proof}
\begin{proposition}\label{prop:QuotientMorphism}
    In the setting of Setting~\ref{set:Kahlers},
    let 
    $F\colon N\rightarrow N'$ be a local diffeomorphism
    preserving the \Kahler~structure, and let
    $\rho\colon \mathbb{A}\rightarrow \mathbb{A}'$ be a surjective
    Lie group homomorphism.
    Suppose that
    $F(N^\circ_{\mathbb{A}})\subset (N')^\circ_{\mathbb{A}'}$
    and 
    $F\circ a=\rho(a)\circ F$ for any $a\in\mathbb{A}$.
    Then there exists a unique map
    $f\colon N^\circ_{\mathbb{A}}/\mathbb{A}\rightarrow
    (N')^\circ_{\mathbb{A}'}/\mathbb{A}'$
    satisfying 
    $f\circ \pi =\pi'\circ F$,
    and~$f$ is a local diffeomorphism preserving the Hessian structure.
\end{proposition}
\begin{proof}
    By the condition 
    $F\circ a=\rho(a)\circ F$, the map 
    $F$ sends
    $\mathbb{A}$-orbits to 
    $\mathbb{A}'$-orbits.
    Hence, there exists a unique map
    $f\colon N^\circ_{\mathbb{A}}/\mathbb{A}\rightarrow
    (N')^\circ_{\mathbb{A}'}/\mathbb{A}'$
    satisfying 
    $f\circ\pi =\pi'\circ F$.

    We show that 
    $f$ is a local diffeomorphism preserving the
    Hessian structure.
    First, the smoothness of 
    $f$ follows from the fact that the
    actions of 
    $\mathbb{A}$ and 
    $\mathbb{A}'$ on
    $N^\circ_{\mathbb{A}}$ and 
    $(N')^\circ_{\mathbb{A}'}$,
    respectively, are free and proper.
    Moreover, 
    $f^*g'=g$ follows from
    Lemma~\ref{lem:KahlerHessianMorphism} (2).
    Thus, it suffices to show that 
    $f$ is a local diffeomorphism
    preserving the connections.
    By Lemma~\ref{lem:Isomorphism} (2), it suffices to show that
    $R:=\rho_*\colon\mathfrak{a}\simarrow\mathfrak{a}'$
    is a Lie algebra isomorphism and that
    $f_*X^B_p=(RX)^B_{f(p)}$ holds for any
    $p\in N^\circ_{\mathbb{A}}/\mathbb{A}$ and any 
    $X\in\mathfrak{a}$.

    First, the invertibility of 
    $R$ follows from the surjectivity
    of 
    $\rho$ and the fact that 
    $\mathbb{A}$ and 
    $\mathbb{A}'$
    are both 
    $n$-dimensional.
    Next, we show that 
    $f_*X^B_p=(RX)^B_{f(p)}$ holds.
    By Lemma~\ref{lem:KahlerHessianMorphism} (1), it suffices to show
    that, for any 
    $p\in N^\circ_{\mathbb{A}}$ and any
    $X\in\mathfrak{a}$,
    $F_*X^V_p=(RX)^V_{F(p)}$ and
    $F_*X^H_p=(RX)^H_{F(p)}$.
    The former follows from Lemma~\ref{lem:Isomorphism} (1).
    Moreover, since 
    $F$ is holomorphic, that is,
    $F_*I=I'F_*$, we have
    $F_*X^H_p=-F_*IX^V_p=-I'(RX)^V_{F(p)}=(RX)^H_{F(p)}$.
\end{proof}
We discuss the case where the same Hessian isomorphism is induced.
\begin{proposition}\label{prop:TwoEquivariantIsomorphisms}
    Consider the setting of Setting~\ref{set:Kahlers}.
    Let 
    $F_1, F_2\colon N\simarrow N'$ be equivariant isomorphisms
    with respect to Lie group isomorphisms
    $\rho_1, \rho_2\colon \mathbb{A}\simarrow\mathbb{A}'$,
    respectively, and suppose that
    $\pi'\circ F_1=\pi'\circ F_2$.
    Then there exists a unique element 
    $a'\in\mathbb{A}'$
    such that 
    $F_2=a'\circ F_1$.
    Moreover, 
    $\rho_1=\rho_2$.
\end{proposition}
\begin{proof}
    By Proposition~\ref{prop:QuotientIsomorphism},
    let 
    $f_1, f_2
    \colon N^\circ_{\mathbb{A}}/\mathbb{A}\rightarrow
    (N')^\circ_{\mathbb{A}'}/\mathbb{A}'$ be the Hessian isomorphisms
    induced by 
    $F_1, F_2$, respectively.
    Then, by the condition 
    $\pi'\circ F_1=\pi'\circ F_2$, we have
    \begin{equation*}
        f_1\circ\pi=\pi'\circ F_1=\pi'\circ F_2
        =f_2\circ\pi,
    \end{equation*}
    and hence 
    $f_1=f_2$.
    Set 
    $f:=f_1=f_2$.

    First, we show that 
    $\rho_1=\rho_2$.
    Setting 
    $R_1:={\rho_1}_*$ and 
    $R_2:={\rho_2}_*$,
    the same argument as in the proof of
    Proposition~\ref{prop:QuotientMorphism} shows that
    $(R_1X)^B=f_*X^B=(R_2X)^B$
    for any 
    $X\in\mathfrak{a}$.
    Hence, 
    $R_1=R_2$, and since 
    $\mathbb{A}$ is connected and abelian,
    we obtain 
    $\rho_1=\rho_2$.
    Set 
    $\rho:=\rho_1=\rho_2$ and 
    $R:=R_1=R_2$.

    Next, we show that there exists a unique 
    $a'\in\mathbb{A}'$
    such that 
    $F_2=a'\circ F_1$.
    The uniqueness is clear.
    We show the existence.
    By Proposition~\ref{prop:QuotientIsomorphism},
    we have 
    $f\circ\pi=\pi'\circ F_1=\pi'\circ F_2$.
    Therefore, 
    $F_1$ and 
    $F_2$ send each 
    $\mathbb{A}$-orbit
    to the same 
    $\mathbb{A}'$-orbit.
    Hence, fixing a point 
    $p\in N^\circ_{\mathbb{A}}$,
    there exists some 
    $a'\in\mathbb{A}'$ such that
    $F_2(p)=a'\cdot F_1(p)$.
    By the same argument as in the proof of
    Proposition~\ref{prop:QuotientMorphism}, we again obtain
    \begin{equation*}
        {F_1}_*X^V=(RX)^V={F_2}_*X^V,\quad
        {F_1}_*X^H=(RX)^H={F_2}_*X^H.
    \end{equation*}
    Therefore, 
    $F_2=a'\circ F_1$.
\end{proof}

Conversely, 
we consider whether an isomorphism between 
Hessian manifolds obtained as quotients 
can be lifted to an equivariant isomorphism between the \Kahler~manifolds.
First, we show that such a lift exists locally:
\begin{lemma}\label{lem:LocalIsomLift}
    Consider Setting~\ref{set:Kahlers}.
    Let 
    $f\colon N^\circ_{\mathbb{A}}/\mathbb{A}
    \simarrow (N')^\circ_{\mathbb{A}'}/\mathbb{A}'$ be an
    isomorphism of Hessian manifolds.
    Then, for any 
    $p\in N^\circ_{\mathbb{A}}$ and
    $p'\in\pi'^{-1}(f(\pi(p)))
    \subset (N')^\circ_{\mathbb{A}'}$,
    there exists a local diffeomorphism 
    $F\colon W\rightarrow
    (N')^\circ_{\mathbb{A}'}$ preserving the \Kahler~structure
    on some connected open neighborhood 
    $W\subset N^\circ_{\mathbb{A}}$
    of 
    $p$, such that 
    $F(p)=p'$ and 
    $f\circ \pi=\pi'\circ F$.
    Moreover, for a fixed~$W$, such an 
    $F$ is unique.
\end{lemma}
\begin{proof}
    Fix 
    $p\in N^\circ_{\mathbb{A}}$ and
    $p'\in\pi'^{-1}(f(\pi(p)))
    \subset (N')^\circ_{\mathbb{A}'}$.
    First, by assumption, we have 
    $f^*g'=g$.
    Hence, by Lemma~\ref{lem:Isomorphism} (2),
    there exists a Lie algebra isomorphism
    $R\colon \mathfrak{a}\simarrow \mathfrak{a}'$
    such that 
    $f_*X^B=(RX)^B$ holds for any 
    $X\in\mathfrak{a}$.

    First, we show the uniqueness of 
    $F$.
    Suppose that there exists a map 
    $F$ satisfying the required
    conditions on a connected open neighborhood
    $W\subset N^\circ_{\mathbb{A}}$ of 
    $p$.
    It suffices to show that, for any 
    $q\in W$ and
    $X\in\mathfrak{a}$,
    $F_*X^V_q=(RX)^V_{F(q)}$ and
    $F_*X^H_q=(RX)^H_{F(q)}$.
    This follows from an argument similar to that in
    Lemma~\ref{lem:KahlerHessianMorphism} (1-b)$\Rightarrow$(1-a).

    Next, we construct~$F$.
    There exist a smooth map 
    $F\colon W\rightarrow W'$,
    defined on connected open neighborhoods
    $W\subset N^\circ_{\mathbb{A}}$ and
    $W'\subset (N')^\circ_{\mathbb{A}'}$ of 
    $p$ and~$p'$,
    respectively, such that 
    $F(p)=p'$ and, for any
    $q\in W$ and 
    $X\in\mathfrak{a}$,
    $F_*X^V_q=(RX)^V_{F(q)}$ and
    $F_*X^H_q=(RX)^H_{F(q)}$.
    Indeed, the frames
    $(X_1^H,\ldots, X_n^H, X_1^V,\ldots, X_n^V)$ and
    $((RX_1)^H,\ldots, (RX_n)^H,\allowbreak (RX_1)^V,\ldots, (RX_n)^V)$
    on 
    $N^\circ_{\mathbb{A}}$ and
    $(N')^\circ_{\mathbb{A}'}$, respectively, associated with
    a basis 
    $(X_1,\ldots, X_n)$ of 
    $\mathfrak{a}$ and the basis
    $(RX_1,\ldots, RX_n)$ of 
    $\mathfrak{a}'$, are integrable
    by Lemma~\ref{lem:LieBrackets}.
    In particular, after shrinking the neighborhoods if necessary,
    we may assume that these frames give coordinate charts on
    connected open neighborhoods 
    $W\subset N^\circ_{\mathbb{A}}$
    and 
    $W'\subset (N')^\circ_{\mathbb{A}'}$ of~$p$ and 
    $p'$,
    respectively.
    The map 
    $F\colon W\rightarrow W'$ induced by the identity map
    on~$\Rn{2n}$ then satisfies the required conditions.
    In particular, 
    $F$ is a local diffeomorphism by construction.

    Finally, we show that 
    $F$ satisfies 
    $f\circ\pi=\pi'\circ F$
    and preserves the \Kahler~structure.
    The equality 
    $f\circ\pi=\pi'\circ F$ follows from the fact that
    both maps
    $f\circ \pi, \pi'\circ F \colon W\rightarrow
    (N')^\circ_{\mathbb{A}'}/\mathbb{A}'$
    send 
    $p$ to 
    $f(\pi(p))$ and send 
    $X^H, X^V$ to
    $(RX)^B, 0$, respectively.
    We show that 
    $F$ preserves the \Kahler~structure, that is,
    $F_*I=I'F_*$ and 
    $F^*h'=h$.
    The equality 
    $F_*I=I'F_*$ follows from
    Lemma~\ref{lem:KahlerHessianMorphism} (1), since
    $F_*X^V_q=(RX)^V_{F(q)}$ and
    $F_*X^H_q=(RX)^H_{F(q)}$ hold for any
    $q\in W$ and 
    $X\in\mathfrak{a}$.
    Moreover, 
    $F^*h'=h$ follows from
    Lemma~\ref{lem:KahlerHessianMorphism} (2).
\end{proof}

We aim to extend Lemma~\ref{lem:LocalIsomLift} 
globally in Proposition~\ref{prop:HessianKahlerIsomorphism}.
To this end, 
we recall that, under certain conditions on the original 
\Kahler~manifold, a local lift can be extended globally.
\begin{fact}[\cite{KN63}]\label{fac:LocalIsometryExtension}
    Let~$N, N'$ be connected, simply connected, complete,
    and analytic Riemannian manifolds.
    Then any local isometry 
    $W\rightarrow N'$ defined on
    a connected open subset 
    $W\subset N$ admits a unique
    global extension.
\end{fact}
\begin{lemma}\label{lem:LocalIsomorphismExtention}
    Let~$N, N'$ be connected, simply connected, complete,
    and analytic \Kahler~manifolds.
    Then any local diffeomorphism 
    $W\rightarrow N'$ preserving
    the \Kahler~structure, defined on a connected open subset
    $W\subset N$, admits a unique global extension.
\end{lemma}
\begin{proof}
    By Fact~\ref{fac:LocalIsometryExtension},
    the given map extends to a global local diffeomorphism
    $F\colon N\rightarrow N'$ preserving the Riemannian metrics.
    Hence, it suffices to show that
    $F_*I=I'F_*$ for the complex structures 
    $I$ and~$I'$
    on 
    $N$ and~$N'$, respectively.
    Let 
    $U\subset N$ be an open subset on which 
    $F|_U$
    is a diffeomorphism, and let 
    $\nabla^{h}, \nabla^{h'}$
    denote the Levi-Civita connections of 
    $h$ and~$h'$, respectively.
    Since~$F|_U$ is an isometry, we have
    $F_*\nabla^h=\nabla^{h'}F_*$.
    On the other hand, by Fact~\ref{fac:KahlerLeviCivita},
    we have
    \begin{equation*}
        \nabla^{h}(I-F_*^{-1}I'F_*)
        =\nabla^hI- F_*^{-1}(\nabla^{h'}I')F_*=0.
    \end{equation*}
    Therefore, the condition 
    $F_*I=I'F_*$ on~$W$
    extends globally on~$N$.
\end{proof}
\begin{fact}[{\cite[Proposition~8.1.4]{Pet06}}]
    \label{fac:KillingLocallity}
    Let~$(N, h)$ be a connected Riemannian manifold and
    let 
    $p\in N$ be arbitrary.
    If Killing vector fields 
    $X, Y$ defined on a neighborhood
    of~$p$ satisfy 
    $X_p=Y_p$ and
    $(\nabla^h X)_p=(\nabla^h Y)_p$, then 
    $X=Y$.
    Here, 
    $\nabla^h$ denotes the Levi-Civita connection of 
    $h$.
    In particular, if two Killing vector fields agree on an
    open subset, then they agree globally.
\end{fact}
\begin{proposition}\label{prop:HessianKahlerIsomorphism}
    In Setting~\ref{set:Kahlers}, suppose further that 
    $N, N'$
    are connected, simply connected, complete, and analytic.
    Let
    $f\colon N^\circ_{\mathbb{A}}/\mathbb{A}\rightarrow
    (N')^\circ_{\mathbb{A}'}/\mathbb{A}'$
    be an isomorphism of Hessian manifolds.
    Then, there exist
    a \Kahler~isomorphism
    $F\colon
    N\simarrow N'$
    and a
    Lie group isomorphism
    $\rho\colon
    \mathbb{A}\simarrow \mathbb{A}'$
    such that
    $F$ is
    an equivariant isomorphism
    with respect to~$\rho$,
    and satisfies
    $f\circ\pi=\pi'\circ F$.
\end{proposition}
\begin{proof}
    Fix 
    $p\in N^\circ_{\mathbb{A}}$ and
    $p'\in\pi'^{-1}(f(\pi(p)))
    \subset (N')^\circ_{\mathbb{A}'}$.
    First, by Lemma~\ref{lem:LocalIsomLift}, there exists a unique
    local diffeomorphism 
    $F\colon W\rightarrow N'$ preserving
    the \Kahler~structure on a connected open neighborhood
    $W\subset N^\circ_{\mathbb{A}}$ of~$p$, such that
    $F(p)=p'$ and 
    $f\circ\pi=\pi'\circ F$ on~$W$.

    First, we show that~$F$ extends to an isomorphism of 
    \Kahler~manifolds satisfying
    $f\circ\pi=\pi'\circ F$.
    By Lemma~\ref{lem:LocalIsomorphismExtention}, 
    $F$ extends to
    a local diffeomorphism 
    $N\rightarrow N'$ preserving the
    \Kahler~structure.
    In this case, applying the same construction to 
    $f^{-1}$,
    we obtain 
    $F^{-1}\colon N'\rightarrow N$.
    In particular, 
    $F$ is an isomorphism.
    Moreover, the Lie algebra isomorphism
    $R\colon \mathfrak{a}\simarrow \mathfrak{a}'$
    defined in the proof of Lemma~\ref{lem:LocalIsomLift}
    satisfies
    $F_*X^V=(RX)^V$ and
    $F_*X^H=(RX)^H$
    on~$W$ for any 
    $X\in\mathfrak{a}$.
    These equalities extend globally.
    Indeed, since 
    $F_*X^V$ and 
    $(RX)^V$ are Killing vector fields,
    Fact~\ref{fac:KillingLocallity} implies that
    $F_*X^V=(RX)^V$ globally on~$N$.
    Since~$F$ is holomorphic, we also have
    $F_*X^H=-F_*IX^V=-I'F_*X^V=-I'(RX)^V=(RX)^H$.
    Finally, the equality 
    $f\circ \pi=\pi'\circ F$ follows from
    the fact that both maps
    $f\circ \pi, \pi'\circ F\colon N^\circ_\mathbb{A}\rightarrow
    (N')^\circ_{\mathbb{A}'}/\mathbb{A}'$
    send 
    $p$ to 
    $f(\pi(p))$ and send 
    $X^H, X^V$ to
    $(RX)^B, 0$, respectively.

    Finally, we show the existence of a Lie group isomorphism
    $\rho\colon\mathbb{A}\simarrow \mathbb{A}'$.
    By Lemma~\ref{lem:Isomorphism} (1), the Lie algebra isomorphism
    $R\colon \mathfrak{a}\simarrow \mathfrak{a}'$
    determines a Lie group homomorphism
    $\rho\colon \mathbb{A}\simarrow \mathbb{A}'$
    satisfying 
    $F\circ a=\rho(a)\circ F$ for any
    $a\in\mathbb{A}$.
    Moreover, applying the same argument to 
    $f^{-1}$,
    we see that $\rho$ is invertible.
\end{proof}

The above discussion yields the following corollary:
\begin{corollary}\label{cor:IsomorphismEquivalence}
    In Setting~\ref{set:Kahlers}, suppose further that 
    $N, N'$
    are connected, simply connected, complete, and analytic.
    Then, the following conditions are equivalent:\\
    (a)~$(N, \mathbb{A}), (N', \mathbb{A}')$ are equivariantly isomorphic.\\
    (b)~$N^\circ_{\mathbb{A}}/\mathbb{A}, (N')^\circ_{\mathbb{A}'}/\mathbb{A}'$
        are isomorphic as Hessian manifolds.
\end{corollary}
\begin{proof}
    The respective implications follow from
    Propositions~\ref{prop:QuotientIsomorphism} and
    \ref{prop:HessianKahlerIsomorphism}.
\end{proof}

Furthermore, we obtain the following proposition:
\begin{proposition}\label{prop:IsomorphismCorrespondence}
    In Setting~\ref{set:Kahlers}, suppose further that 
    $N=N'$.
    Set
    \begin{align*}
        \mathrm{Aut}_{\mathbb{A}}(N)&:=\{(F\colon N\simarrow N,~
        \rho\colon \mathbb{A}\simarrow\mathbb{A})
        ~:~\textrm{Equivariant isomorphism}\}\\
        \mathrm{Aut}(N^\circ_{\mathbb{A}}/\mathbb{A})
        &:=\{f\colon N^\circ_{\mathbb{A}}/\mathbb{A}\simarrow
    N^\circ_{\mathbb{A}}/\mathbb{A}~:~
    \textrm{Hessian isomorphism}\}.
    \end{align*}
    Then, the following map
    \begin{equation}\label{eq:IsomorphismCorrespondence}
        \sigma\colon \mathrm{Aut}_{\mathbb{A}}(N)\rightarrow  
        \mathrm{Aut}(N^\circ_{\mathbb{A}}/\mathbb{A}),~
        (F,\rho)\mapsto f
    \end{equation}
    defined by Proposition~\ref{prop:QuotientIsomorphism}
    is a group homomorphism, and its kernel is
    $\ker \sigma=\mathbb{A}$.
    Here, by Note~\ref{not:EquivariantIsomorphism} (2),
    $\mathbb{A}$ is naturally embedded in
    $\mathrm{Aut}_{\mathbb{A}}(N)$.
    Moreover, if 
    $N$ is connected, simply connected, complete,
    and analytic, then 
    $\sigma$ is surjective.
    In particular, we have
    \begin{equation*}
        \mathrm{Aut}_{\mathbb{A}}(N)/\mathbb{A}\cong 
    \mathrm{Aut}(N^\circ_{\mathbb{A}}/\mathbb{A}).
    \end{equation*}
\end{proposition}
\begin{proof}
    The fact that 
    $\sigma$ is a homomorphism follows from
    Proposition~\ref{prop:QuotientIsomorphism}.
    Moreover, 
    $\ker \sigma=\mathbb{A}$ follows from
    Proposition~\ref{prop:TwoEquivariantIsomorphisms}.
    Finally, the surjectivity of 
    $\sigma$ when 
    $N$ is connected,
    simply connected, complete, and analytic follows from
    Proposition~\ref{prop:HessianKahlerIsomorphism}.
\end{proof}

\section{Correspondence between subgroups and standard models}
\label{Sec:Correspondence}
In this section, 
we show that the standard models (Example~\ref{exm:StandardModels}) 
are realized as quotients arising from explicit group actions 
on specific \Kahler~manifolds.
We also prove that every Hessian manifold of 
constant Hessian sectional curvature admits 
a local open embedding into a standard model 
(Proposition~\ref{prop:LocalHessianEmbedding}).

First, we show that the classification of Hessian manifolds
can be reduced to the classification of \Kahler~manifolds
together with group actions.
\begin{note}\label{not:ClosedSubgroup}
    Let~$N=(N,I,h,\omega)$ be a 
    $2n$-dimensional connected,
    simply connected, complete, and analytic \Kahler~manifold,
    and denote its automorphism group as a \Kahler~manifold by
    $G:=\Aut(N)$.
    An effective proper action on 
    $N$ by a Lie group~$H$
    preserving the \Kahler~structure embeds 
    $H$ as a closed
    subgroup of~$G$.
    Conversely, a closed subgroup~$H\subset G$ naturally defines
    an effective proper action on 
    $N$ preserving the
    \Kahler~structure (\cite[Proposition 3.62]{AB15}).
\end{note}
\begin{proposition}
    Let~$N=(N,I,h,\omega)$ be a~$2n$-dimensional 
    connected, simply connected, complete, and analytic 
    \Kahler~manifold,
    and denote the automorphism group of~$N$
    by~$G:=\Aut(N)$.
    Suppose that an~$n$-dimensional connected abelian closed subgroup
    $\mathbb{A}, \mathbb{A}'\subset G$
    defines (effective proper) Hamiltonian actions on~$N$.
    Then, the following conditions are equivalent:\\
    (a)~$\mathbb{A}$ and~$\mathbb{A}'$ are conjugate in~$G$.\\
    (b)~$(N,\mathbb{A})$ and~$(N',\mathbb{A}')$ are equivariantly isomorphic.\\
    (c)~The quotient Hessian manifolds
    $N^\circ_{\mathbb{A}}/\mathbb{A}, 
    N^\circ_{\mathbb{A}'}/\mathbb{A}'$
    are isomorphic as Hessian manifolds.
\end{proposition}
\begin{proof}
    The equivalence~$(a)\Leftrightarrow(b)$ follows 
    from the fact that the existence of~$F\in G$ satisfying
    $F\mathbb{A}F^{-1}=\mathbb{A}'$
    is equivalent to the fact that~$F$ is 
    an equivariant isomorphism with respect to the Lie group isomorphism
    $\rho\colon \mathbb{A}\rightarrow\mathbb{A}',~
    a\mapsto FaF^{-1}$.
    The equivalence~$(b)\Leftrightarrow(c)$ follows 
    from Corollary~\ref{cor:IsomorphismEquivalence}.
\end{proof}

In Section~\ref{Sec:Correspondence}, 
we consider actions on connected, simply connected, complete and 
analytic \Kahler~manifolds as defined below, 
together with their quotients.
\begin{example}[\cite{KN69}]\label{exm:ComplexSpaceForms}
    (1)~Let~$C>0$.
    The complex projective space
    $\mathbb{C}P^n:=(\mathbb{C}^n\setminus\{0\})/\mathbb{C}^\times$
    is naturally a \Kahler~manifold
    with respect to the Riemannian metric~$h_{\mathrm{std}}$,
    called the Fubini-Study metric.
    In the complex chart
    on~$U_{n+1}:=\{[\alpha^1:\cdots :\alpha^{n+1}]
    \in\mathbb{C}P^n~:~\alpha^{n+1}\neq 0\}$
    given by
    $U_{n+1}\rightarrow \Cn{n},~[\alpha^1:\cdots :\alpha^{n+1}]\mapsto 
    z:=\left(\frac{\alpha^1}{\alpha^{n+1}}
    ,\ldots, \frac{\alpha^n}{\alpha^{n+1}}\right)$,
    $h_{\mathrm{std}}$ is expressed as
    \begin{equation*}
        h_{\mathrm{std}}=2\left(\frac{\delta_{ij}}{1+\|z\|^2}
        -\frac{\bar{z}^iz^j}{(1+\|z\|^2)^2}\right)
        (dz^i \otimes d\bar{z}^j+d\bar{z}^j\otimes dz^i).
    \end{equation*}
    Moreover, the function~$\Psi\colon U_{n+1}\rightarrow \Rn{}$
    given by
    \begin{equation*}
        \Psi=4\log (1+\|z\|^2)
    \end{equation*}
    is a \Kahler~potential on~$U_{n+1}$.\\
    (2)~The complex Euclidean space~$\Cn{n}$
    is a \Kahler~manifold
    with respect to the
    Euclidean metric~$h_{\mathrm{std}}$.
    Moreover, the function~$\Psi\colon \Cn{n}\rightarrow \Rn{}$
    given by
    \begin{equation*}
        \Psi=4\|z\|^2
    \end{equation*}
    is a \Kahler~potential on~$\Cn{n}$.\\
    (3)~Let~$C<0$.
    The open unit ball~$D_n=\{z\in \Cn{n}~:~\|z\|<1\}$
    is a \Kahler~manifold
    with respect to the
    Riemannian metric~$h_{\mathrm{std}}$ given by
    \begin{equation*}
        h_\mathrm{std}=2\left(\frac{\delta_{ij}}{1-\|z\|^2}
        -\frac{\bar{z}^iz^j}{(1-\|z\|^2)^2}\right)
        (dz^i \otimes d\bar{z}^j+d\bar{z}^j\otimes dz^i).
    \end{equation*}
    Moreover, the function~$\Psi\colon D_n\rightarrow \Rn{}$
    given by
    \begin{equation*}
        \Psi=-4\log (1-\|z\|^2)
    \end{equation*}
    is a \Kahler~potential on~$D_n$.\\
\end{example}

Using the results established in the subsequent sections, 
we prove the following proposition.
\begin{proposition}\label{prop:Section4}
    Let $c\in\Rn{}$.\\
    (1)~The \Kahler~manifold
    $N:=\begin{cases}
    \mathbb{C}P^n &(c<0)\\
    \mathbb{C}^n  &(c=0)\\
    D_n  & (c>0)
    \end{cases}$
    with metric
    $\begin{cases}
    \frac{1}{|c|}h_{\mathrm{std}} &(c<0)\\
    h_{\mathrm{std}} &(c=0)\\
    \frac{1}{|c|}h_{\mathrm{std}}  & (c>0)
    \end{cases}$
    has automorphism group
    $G:=\Aut(N)\cong\begin{cases}
    PU(n+1) &(c<0)\\
    U(n)\ltimes \Cn{n} &(c=0)\\
    PU(n,1)  & (c>0)
    \end{cases}$.
    Every~$n$-dimensional connected abelian 
    subgroup~$\mathbb{A}\subset G$ 
    whose action on~$N$ generates an isotropic subspace 
    at some point in~$N$ is conjugate to one of
    $\begin{cases}
        \mathbb{A}_-^n &(c<0)\\
        \mathbb{A}_0^{n,k}~(k=0,\ldots, n)&(c=0)\\
        \mathbb{A}_+^{n},
        \mathbb{A}_+^{n,\mathrm{H}},
        \mathbb{A}_+^{n,k}~(k=1,\ldots,n)& (c>0)
    \end{cases}$
    and, in particular,~$\mathbb{A}$ is a closed subgroup.\\
    (2)~The subgroup~$\mathbb{A}\subset G$ given above induces an
    effective proper Hamiltonian action on~$N$, and the
    quotient manifold is, respectively, isomorphic as a Hessian manifold to
    $N^\circ_{\mathbb{A}}/\mathbb{A}\cong\begin{cases}
        \mathcal{S}_-^n(c) &(c<0)\\
        \mathcal{R}_0^{n,k}(0)~(k=0,\ldots, n)&(c=0)\\
        \mathcal{H}_+^{n}(c),
        \mathcal{H}_+^{n,\mathrm{H}}(c),
        \mathcal{H}_+^{n,k}(c)~(k=1,\ldots,n)& (c>0)
    \end{cases}$.
\end{proposition}
\begin{proof}
    (1)~For~$c>0$,~$c=0$, and~$c<0$, 
    the claim is the same statement as 
    Note~\ref{not:PositiveAbelianSubgroup}, 
    Lemma~\ref{lem:ZeroAbelianSubgroup} and 
    Note~\ref{fac:NegativeAbelianSubgroup}, respectively.

(2)~For~$c>0$,~$c=0$, and~$c<0$, 
    the claim is the same statement as  
    Propositions~\ref{prop:PositiveQuotient}, 
    \ref{prop:ZeroQuotient} and \ref{prop:NegativeQuotient}, respectively.
\end{proof}

\subsection{$N=\mathbb{C}P^n$}
In this section, we determine all possible connected abelian closed 
subgroups~$\mathbb{A}\subset \Aut(N)$ 
for~$N=\mathbb{C}P^n$
and show that the resulting Hessian manifolds 
obtained as quotients are isomorphic to one of the standard models.

\begin{note}
    For the \Kahler~manifold~$N=\mathbb{C}P^n$, 
    the automorphism group~$G$ 
    is given by the action of~$PU(n+1)$ 
    on~$\mathbb{C}P^n$ induced by matrix multiplication.
\end{note}
\begin{definition}\label{def:PositiveAbelianSubgroups}
    We define an~$n$-dimensional connected abelian closed subgroup
    $\mathbb{A}_-^n\subset PU(n+1)$
    as follows:
    \begin{equation*}
        \mathbb{A}_-^n
        :=\left\{
            \left[\begin{pmatrix}
                a&  \\
                & 1
            \end{pmatrix}\right]\in PU(n+1)
        ~:~
        a\in U(1)^{n}\right\}.
    \end{equation*}
    Then,~$\mathfrak{a}_-^n:=\mathrm{Lie}(\mathbb{A}_-^n)$
    is given by
    \begin{equation*}
        \mathfrak{a}_-^n
        =\left\{
            \left[\begin{pmatrix}
                X&  \\
                & 0
            \end{pmatrix}\right]
        ~:~
        X\in \mathfrak{u}(1)^{n} \right\}.
    \end{equation*}
\end{definition}
\begin{note}\label{not:PositiveAbelianSubgroup}
    $\mathbb{A}_-^n$ is the unique maximal torus of~$PU(n+1)$.
    Therefore, every~$n$-dimensional 
    connected abelian subgroup of~$PU(n+1)$ 
    is conjugate to~$\mathbb{A}_-^n$.
\end{note}

\begin{lemma}\label{lem:PositiveQuotient}
    Let~$C>0$. Set~$N:=\mathbb{C}P^n$ equipped with
    the metric~$\frac{1}{C}h_{\mathrm{std}}$,
    and let
    $\mathbb{A}:=\mathbb{A}_-^n\subset PU(n+1)$.
    Then the map~$\mu\colon N\rightarrow
    \mathfrak{a}^*\cong\mathbb{R}^n$ defined by
    \begin{equation*}
        \mu([\alpha^1:\cdots:\alpha^{n+1}])
        :=-\frac{1}{C\|\alpha\|^2}
        (|\alpha^1|^2,\ldots,|\alpha^n|^2)
    \end{equation*}
    is a moment map.
    Here, we identify~$\mathfrak{a}$ with~$\mathbb{R}^n$
    via the basis
    $\left(\frac{\iu}{2}E_{11},\ldots,
    \frac{\iu}{2}E_{nn}\right)$ of~$\mathfrak{a}$,
    where~$E_{ii}$ denotes the matrix unit.
\end{lemma}
\begin{proof}
    It suffices to consider the case~$C=1$.
    Moreover, by Lemma~\ref{lem:OpenDense}, 
    it suffices to show that
    the restriction of~$\mu$ to
    $N^\circ_\mathbb{A}=\{[\alpha^1:\cdots:\alpha^{n+1}]\in N ~:~
    \forall i,~\alpha^i\neq 0\}$
    is a moment map.
    We use Proposition~\ref{prop:MomentMaps}.

    For the Hessian manifold~$\mathcal{S}^n(-1)$ defined in
    Example~\ref{exm:StandardModels}, let
    $N':=\mathcal{S}^n(-1)\times\mathbb{R}^n$ be equipped with the
    \Kahler~structure in Fact~\ref{fac:TangentBundleKahler} and the
    effective proper Hamiltonian~$\mathbb{R}^n$-action in
    Proposition~\ref{prop:TangentMomentMap}.
    We denote by~$\Psi'$ and~$\mu'$ the \Kahler~potential and moment map,
    respectively, on~$N'$ defined in
    Example~\ref{exm:StandardModelsKahler}.
    Using the complex chart~$(w^1,\ldots,w^n)$ on~$N'$ from
    Note~\ref{not:TangentHolomorphicChart}, define
    \begin{equation*}
        F\colon N'\rightarrow N^\circ_\mathbb{A}
        \subset\mathbb{C}P^n,~
        (w^1,\ldots, w^n)\mapsto \left[e^\frac{w^1}{2}:\cdots
        :e^{\frac{w^n}{2}}:1\right].
    \end{equation*}
    We also define the surjective homomorphism
    $\rho\colon\mathbb{R}^n\rightarrow\mathbb{A}$ by
    \begin{equation*}
        \rho\colon \Rn{n}\rightarrow \mathbb{A}\subset PU(n+1),~
        (a^1,\ldots,a^n)\mapsto 
        \left[\begin{pmatrix}
            e^{ \frac{\iu}{2}a^1}& & & \\
             &\ddots& & \\
             & &e^{\frac{\iu}{2}a^n}&\\
             & & &1
        \end{pmatrix}
        \right].
    \end{equation*}
    It therefore suffices to show that~$F$ is a surjective local
    diffeomorphism preserving the \Kahler~structures, and that
    $F$,~$\rho$,~$\mu'$ and~$\mu$ satisfy conditions~(i) and~(ii) of
    Proposition~\ref{prop:MomentMaps}.

    Conditions~(i) and~(ii) follow by direct computation.
    Moreover,~$N^\circ_{\mathbb{A}}\subset U_{n+1}$ 
    (see Example~\ref{exm:ComplexSpaceForms}), and the \Kahler~potential
    $\Psi:=4\log(1+\|z\|^2)$
    on~$U_{n+1}\subset N$ satisfies
    $\Psi\circ F-\Psi'=0$.
    Hence, by Note~\ref{not:KahlerPotentialDifference},
    $F$ preserves the \Kahler~structures.
    The inverse of~$F$ is locally given by
    \begin{equation*}
        N^\circ_\mathbb{A}\rightarrow N',~
        (z^1,\ldots, z^n)\mapsto (2\log z^1,\ldots, 2\log z^n),
    \end{equation*}
    where a branch of~$\log$ is chosen appropriately. 
    Therefore,~$F$ is a surjective local diffeomorphism.
\end{proof}
\begin{proposition}\label{prop:PositiveQuotient}
    Let~$C>0$.
    Set~$N:=\mathbb{C}P^n$ with~$\frac{1}{C}h_{\mathrm{std}}$ and
    $\mathbb{A}:=\mathbb{A}_-^n\subset PU(n+1)$.\\
    (1)~$\mathbb{A}$ naturally induces a Hamiltonian action on~$N$.\\
    (2)~As Hessian manifolds,~$N_\mathbb{A}^\circ/\mathbb{A}
    \cong \mathcal{S}^n(-C)$. 
\end{proposition}
\begin{proof}
    We use the maps~$F$,~$\rho$,~$\mu$ and~$\mu'$ defined in
    Lemma~\ref{lem:PositiveQuotient} and its proof.

    (1) follows immediately from Lemma~\ref{lem:PositiveQuotient}.
    For (2), 
    by Proposition~\ref{prop:QuotientMorphism}, 
    it suffices to show
    that the local diffeomorphism
    $f\colon \mathcal{S}^n(-C)\rightarrow N_\mathbb{A}^\circ/\mathbb{A}$
    induced by~$F$ is bijective.

    Surjectivity follows from the surjectivity of~$F$.
    To prove injectivity, by Proposition~\ref{prop:QuotientMorphism}, 
    we have
    $\pi(F(p,0))=f(p)$.
    Thus, for arbitrary $p,q\in\mathcal{S}^n(-C)$, if
    $f(p)=f(q)$, then
    $\pi(F(p,0))=\pi(F(q,0))$.
    Since~$\mu$ is~$\mathbb{A}$-invariant, this implies
    $\mu(F(p,0))=\mu(F(q,0)).$
    On the other hand, the map
    $\mathcal{S}^n(-C)\rightarrow
    \mathfrak{a}^*,~p\mapsto \mu(F(p,0))$
    is given by
    \begin{equation*}
        (\theta^1,\ldots,\theta^n)\mapsto-\frac{1}{C}\left(\frac{e^{\theta^1}}{1+e^{\theta^1}
    +\cdots+e^{\theta^n}},\ldots,
    \frac{e^{\theta^n}}{1+e^{\theta^1}
    +\cdots+e^{\theta^n}}\right),
    \end{equation*}
    which is injective by a direct computation. 
    Hence~$f$ is injective. Therefore,~$f$ is bijective.
\end{proof}

\subsection{$N=\mathbb{C}^n$}
In this section, we consider the \Kahler~manifold
$N:=\mathbb{C}^n$ equipped with the Euclidean metric.
We show that the connected abelian closed subgroups
$\mathbb{A}_0^{n,k}$ ($k=0,\ldots,n$) of the automorphism group,
given in Definition~\ref{def:ZeroAbelianSubgroups}, 
define 
effective proper Hamiltonian actions,
and that the
Hessian manifolds obtained as the corresponding quotients are
isomorphic to standard models.
We also show that the subgroups~$\mathbb{A}_0^{n,k}$ are the only
$n$-dimensional connected abelian subgroups that give rise to
Hamiltonian actions.

\begin{note}
    For the \Kahler~manifold~$N=\mathbb{C}^n$ 
    equipped with the Euclidean metric, 
    the automorphism group is
    $G\cong U(n)\ltimes\mathbb{C}^n$, 
    acting by affine transformations.
\end{note}
\begin{definition}\label{def:ZeroAbelianSubgroups}
    For~$k=0, \ldots, n$,
    we define~$n$-dimensional connected abelian closed
    subgroups
    $\mathbb{A}_0^{n,k}$ as follows:
    \begin{equation*}
        \mathbb{A}_0^{n,k}
        :=\left\{\left(
            \begin{pmatrix}
                1&  \\
                & a \\
            \end{pmatrix},
            \begin{pmatrix}
                b\\
                0
            \end{pmatrix}
            \right)
            \in U(n)\ltimes\Cn{n}
        ~:~
        a\in U(1)^{n-k},~b\in \Rn{k} \right\}.
    \end{equation*}
    Then,~$\mathfrak{a}_0^{n,k}:=\mathrm{Lie}(\mathbb{A}_0^{n,k})$
    is given by
    \begin{equation*}
        \mathfrak{a}_0^{n,k}=
        \left\{\left(
            \begin{pmatrix}
            0&  \\
                & X \\
            \end{pmatrix},
            \begin{pmatrix}
                y\\
                0
            \end{pmatrix}
            \right)
        ~:~
        X\in\mathfrak{u}(1)^{n-k},~y\in \Rn{k} \right\}.
    \end{equation*}
\end{definition}
\begin{lemma}\label{lem:ZeroQuotient}
    Let~$N:=\Cn{n}$ with Euclidean metric,
    and let~$\mathbb{A}:=\mathbb{A}_0^{n,k}$.
    Then, the map~$\mu\colon N\rightarrow \mathfrak{a}^*\cong\Rn{n}$
    defined by
    \begin{equation*}
        \mu(z^1,\ldots,z^n):=(2\Im z^1,\ldots,2\Im z^k, -|z^{k+1}|^2,
        \ldots,-|z^n|^2)
    \end{equation*}
    is a moment map.
    Here, we identify~$\mathfrak{a}$ with~$\Rn{n}$
    via the basis
    \begin{equation*}
        \left(
    \frac{1}{2}
    (0, e_1),\ldots
    \frac{1}{2}(0, e_k),
    \frac{\iu}{2}\left(
        E_{(k+1)(k+1)}, 0
    \right),
    \ldots,
    \frac{\iu}{2}\left(E_{nn}, 0
    \right)\right)
    \end{equation*}
    of~$\mathfrak{a}$,
    where~$e_i$ denotes the standard basis of~$\Rn{n}$,
    and~$E_{ii}$ denotes the matrix unit.
\end{lemma}
\begin{proof}
    By Lemma~\ref{lem:OpenDense},
    it suffices to show that the
    restriction of~$\mu$ to~$N_\mathbb{A}^\circ=\{(z^1,\ldots,z^n)
    \in\Cn{n}~:~ z^{k+1},\ldots, z^n\neq 0\}$
    is a moment map.
    We use Proposition~\ref{prop:MomentMaps}.

    For the Hessian manifold~$\mathcal{R}^{n,k}(0)$ defined in
    Example~\ref{exm:StandardModels}, let
    $N':=\mathcal{R}^{n,k}(0)\times\mathbb{R}^n$ be equipped with the
    \Kahler~structure in Fact~\ref{fac:TangentBundleKahler} and the
    effective proper Hamiltonian~$\mathbb{R}^n$-action in
    Proposition~\ref{prop:TangentMomentMap}.
    We denote by~$\Psi'$ and~$\mu'$ the \Kahler~potential and moment map,
    respectively, on~$N'$ defined in
    Example~\ref{exm:StandardModelsKahler}.
    Using the complex chart~$(w^1,\ldots,w^n)$ on~$N'$ from
    Note~\ref{not:TangentHolomorphicChart}, define
    \begin{equation*}
        F\colon N'\rightarrow N^\circ_\mathbb{A}
        \subset \Cn{n},~
        (w^1,\ldots, w^n)\mapsto \left(-\frac{\iu}{2}w^1,\ldots,
        -\frac{\iu}{2}w^k, e^\frac{w^{k+1}}{2},\ldots,
        e^\frac{w^{n}}{2}\right).
    \end{equation*}
    We also define the surjective homomorphism
    $\rho\colon\mathbb{R}^n\rightarrow\mathbb{A}$ by
    \begin{equation*}
        \rho\colon \Rn{n}\rightarrow \mathbb{A},~
        (a^1,\ldots,a^n)\mapsto 
        \left(
        \begin{pmatrix}
            1 & & & & & \\
             & \ddots& && & \\
             & &1 & & & \\
             & & &e^{\frac{\iu}{2}a^{k+1}}& & \\
             & & & &\ddots& \\
             & & & & &e^{\frac{\iu}{2}a^n} 
        \end{pmatrix},
        \begin{pmatrix}
            \frac{a^1}{2}\\ \vdots\\ \frac{a^k}{2}\\
            0 \\ \vdots\\ 0
        \end{pmatrix}
        \right).
    \end{equation*}
    It therefore suffices to show that~$F$ is a surjective local
    diffeomorphism preserving the \Kahler~structures, and that
    $F$,~$\rho$,~$\mu'$ and~$\mu$ satisfy conditions~(i) and~(ii) of
    Proposition~\ref{prop:MomentMaps}.

    Conditions~(i) and~(ii) follow by direct computation.
    Moreover,~the \Kahler~potential
    $\Psi:=4\|z\|^2$
    on~$N$ satisfies
    \begin{equation*}
        \Psi\circ F-\Psi'=-\frac{1}{2}((w^1)^2+\cdots+(w^k)^2)
        -\frac{1}{2}
        ((\bar{w}^1)^2+\cdots+(\bar{w}^k)^2)
    \end{equation*}.
    Hence, by Note~\ref{not:KahlerPotentialDifference},
    $F$ preserves the \Kahler~structures.
    The inverse of~$F$ is locally given by
    \begin{equation*}
        N^\circ_\mathbb{A}\rightarrow N',~
        (z^1,\ldots, z^n)\mapsto (2\iu z^1,\ldots, 2\iu z^k,
        2\log z^{k+1},\ldots, 2\log z^n).
    \end{equation*}
    where a branch of~$\log$ is chosen appropriately. 
    Therefore,~$F$ is a surjective local diffeomorphism.
\end{proof}
\begin{proposition}\label{prop:ZeroQuotient}
    Set~$k=0,\ldots, n$,
    Let~$N:=\Cn{n}$ with Euclidean metric,
    and let
    $\mathbb{A}:=\mathbb{A}_0^{n,k}\subset U(n)\ltimes\Cn{n}$.\\
    (1)~$\mathbb{A}$ naturally induces a Hamiltonian action on~$N$.\\
    (2)~As Hessian manifolds,~$
    N_\mathbb{A}^\circ/\mathbb{A}
    \cong \mathcal{R}^{n,k}(0)$. 
\end{proposition}
\begin{proof}
    We use the maps~$F$,~$\rho$,~$\mu$ and~$\mu'$ defined in
    Lemma~\ref{lem:ZeroQuotient} and its proof.

    (1) follows immediately from Lemma~\ref{lem:ZeroQuotient}.
    For (2), 
    by Proposition~\ref{prop:QuotientMorphism}, 
    it suffices to show
    that the local diffeomorphism
    $f\colon \mathcal{R}^{n,k}(0)\rightarrow N_\mathbb{A}^\circ/\mathbb{A}$
    induced by~$F$ is bijective.

    Surjectivity follows from the surjectivity of~$F$.
    To prove injectivity, by Proposition~\ref{prop:QuotientMorphism}, 
    we have
    $\pi(F(p,0))=f(p)$.
    Thus, for arbitrary $p,q\in\mathcal{R}^{n,k}(0)$, if
    $f(p)=f(q)$, then
    $\pi(F(p,0))=\pi(F(q,0))$.
    Since~$\mu$ is~$\mathbb{A}$-invariant, this implies
    $\mu(F(p,0))=\mu(F(q,0)).$
    On the other hand, the map
    $\mathcal{R}^{n,k}(0)\rightarrow
    \mathfrak{a}^*,~p\mapsto \mu(F(p,0))$
    is given by
    \begin{equation*}
        (\theta^1,\ldots,\theta^n)\mapsto -\left(\theta^1,\ldots,\theta^k,
    e^{\theta^{k+1}},\ldots, e^{\theta^n} \right)
    \end{equation*}
    which is injective by a direct computation. 
    Hence~$f$ is injective. Therefore,~$f$ is bijective.
\end{proof}

We then show that any~$n$-dimensional connected abelian closed subgroup
$\mathbb{A}\subset U(n)\ltimes\mathbb{C}^n$ whose action is Hamiltonian
must be one of the subgroups~$\mathbb{A}^{n,k}_0$.
\begin{lemma}\label{lem:LagrangianSubspace}
    For any isotropic subspace~$V\subset\mathbb{C}^n$ 
    of the real vector space~$\mathbb{C}^n$ 
    with respect to the standard symplectic form
    $\omega\colon (v,w)\mapsto -\Im v^*w$,
    that is, a real vector subspace~$V\subset\Cn{n}$
    satisfying~$\omega|_V=0$, 
    there exists some $a\in U(n)$ such that
    $aV\subset \Rn{n}\subset \Cn{n}$.
\end{lemma}
\begin{proof}
    Let~$V$ be an isotropic subspace. 
    Then~$\iu V\cap V={0}$. 
    Hence, an~$\mathbb{R}$-orthonormal basis 
    $(v_1,\ldots,v_m)$ of~$V$ with respect to 
    the standard real inner product~$(v,w)\mapsto \Re v^*w$ 
    is also a~$\mathbb{C}$-orthonormal system in~$\mathbb{C}^n$ 
    with respect to the standard Hermitian inner product 
    $(v,w)\mapsto v^*w$. Here,~$m:=\dim_{\mathbb{R}}V$.
    Thus, choose~$v_{m+1},\ldots,v_n\in\mathbb{C}^n$ 
    such that~$(v_1,\ldots,v_n)$ is a~$\mathbb{C}$-orthonormal 
    basis of~$\mathbb{C}^n$ with respect to the 
    standard Hermitian inner product, 
    and set~$u:=(v_1,\ldots,v_n)$. 
    Then~$u\in U(n)$ and~$u^{-1}V\subset\mathbb{R}^n$.
\end{proof}
\begin{note}
    For the action of~$G=U(n)\ltimes\mathbb{C}^n$ 
    on~$\mathbb{C}^n$, 
    the vector field generated by~$(X,y)\in\mathfrak{g}$ 
    is given, at~$z\in\mathbb{C}^n$, 
    by~$Xz+y$ 
    nder the natural identification~$T_z\mathbb{C}^n\cong\mathbb{C}^n$.
\end{note}

\begin{lemma}\label{lem:ZeroAbelianSubgroup}
    Let~$\mathbb{A}\subset G:=U(n)\ltimes\mathbb{C}^n$ 
    be an $n$-dimensional connected abelian subgroup. 
    If the action of~$\mathbb{A}$ generates 
    an isotropic subspace 
    at some point of~$N=\mathbb{C}^n$, 
    then~$\mathbb{A}$ is conjugate to 
    $\mathbb{A}_0^{n,k}$ for some $k=0,\ldots,n$.
\end{lemma}
\begin{proof}
    It suffices to show that the Lie algebra~$\mathfrak{a}$ 
    is conjugate to~$\mathfrak{a}_0^{n,k}$.
    Let~$p_1:\mathfrak{g}\rightarrow\mathfrak{u}(n)$ and 
    $p_2:\mathfrak{g}\rightarrow\mathbb{C}^n$ 
    denote the projections onto the 
    $\mathfrak{u}(n)$ and~$\mathbb{C}^n$, respectively. 
    Then,~$p_1$ is a Lie algebra homomorphism.

    First, we show the following claim:\\
    \textbf{Claim.}
    For any~$n$-dimensional 
    abelian subalgebra~$\mathfrak{a}\subset\mathfrak{g}$,
    there exists some~$0\leq k\leq n$ 
    such that, up to the adjoint action,
    \begin{equation*}
        \mathfrak{a}\subset 
    \left\{\left(\begin{pmatrix}
        0&\\ & D
    \end{pmatrix},
    \begin{pmatrix}
        y\\0
    \end{pmatrix}
    \right)~:~ D\in\mathfrak{u}(1)^{n-k},~y\in \Cn{k}\right\}.
    \end{equation*}

    Since~$p_1(\mathfrak{a})\subset\mathfrak{u}(n)$ 
    is an abelian subalgebra, 
    it can be simultaneously diagonalized. 
    Thus, for some~$a\in U(n)$, 
    there exists $0\leq k\leq n$ such that
    $Ad_{a}p_1(\mathfrak{a})=
       \left\{\begin{pmatrix}
        0&\\ & D
    \end{pmatrix}~:~ D\in\mathfrak{u}(1)^{n-k}\right\}.$
    Let~$g_1:=(a,0)\in G$. Then
    \begin{equation*}
        \mathfrak{a}':=Ad_{g_1}\mathfrak{a}
        \subset
        \mathfrak{h}:=\left\{\left(\begin{pmatrix}
            0& \\ &D
        \end{pmatrix},y
        \right)~:~D\in\mathfrak{u}(1)^{n-k},~y\in \Cn{n}\right\}.
    \end{equation*}
    Next, fix an element
    $X:=\left(\begin{pmatrix}
            0& \\ &S
        \end{pmatrix},
        \begin{pmatrix}
            t\\ t'
        \end{pmatrix}
    \right)\in \mathfrak{a}'$
    such that~$S$ is invertible, and set
     $g_2:=\left(\begin{pmatrix}
                1& \\
                 &1
            \end{pmatrix},
            \begin{pmatrix}
                0\\ -S^{-1}t'
            \end{pmatrix}\right)
        \in G$.
    Then
    $Ad_{g_2}X=\left(\begin{pmatrix}
            0& \\ &S
        \end{pmatrix},
        \begin{pmatrix}
            t\\ 0
        \end{pmatrix}
    \right)$.
    Therefore,
    \begin{equation*}
        Ad_{g_2}\mathfrak{a}'
        \subset
        \mathfrak{z}_{\mathfrak{h}}(Ad_{g_2}X)
        =\left\{\left(\begin{pmatrix}
        0&\\ & D
    \end{pmatrix},
    \begin{pmatrix}
        y\\0
    \end{pmatrix}
    \right)~:~ D\in\mathfrak{u}(1)^{n-k},~y\in \Cn{k}\right\}.
    \end{equation*}
    This proves the claim.
    
    Let~$\mathfrak{a}:=\mathrm{Lie}(\mathbb{A})$. 
    By the claim, we may assume that, for some~$0\leq k\leq n$,
    \begin{equation*}
        \mathfrak{a}\subset 
    \left\{\left(\begin{pmatrix}
        0&\\ & D
    \end{pmatrix},
    \begin{pmatrix}
        y\\0
    \end{pmatrix}
    \right)~:~ D\in\mathfrak{u}(1)^{n-k},~y\in \Cn{k}\right\}.
    \end{equation*}
    We show that~$\mathfrak{a}$ is conjugate to~$\mathfrak{a}_0^{n,k}$ 
    using Lemma~\ref{lem:LagrangianSubspace}.
    It suffices to show, up to the adjoint action, 
    that~$p_2(\mathfrak{a})\subset\mathbb{R}^k$.
    By assumption, the action of~$\mathbb{A}$ 
    generates an isotropic subspace 
    at some point~$z\in\mathbb{C}^n$. 
    Thus, for arbitrary
    $X=\left(\begin{pmatrix}
        0&\\ & D
    \end{pmatrix},
    \begin{pmatrix}
        y\\0
    \end{pmatrix}
    \right),~
    X'=\left(\begin{pmatrix}
        0&\\ & D'
    \end{pmatrix},
    \begin{pmatrix}
        y'\\0
    \end{pmatrix}
    \right)\in \mathfrak{a},$
    the vector fields~$X^V$ and~${X'}^V$ 
    generated on~$N=\mathbb{C}^n$ satisfy
    \begin{equation*}
        0=\omega(X^V, {X'}^V)
        =-\Im\left(\begin{pmatrix}
        0&\\ & D
    \end{pmatrix}z+
    \begin{pmatrix}
        y\\0
    \end{pmatrix}\right)^*\left(\begin{pmatrix}
        0&\\ & D'
    \end{pmatrix}z+
    \begin{pmatrix}
        y'\\0
    \end{pmatrix}\right)
        =-\Im y^*y'.
    \end{equation*}
    Hence,~$p_2(\mathfrak{a})\subset\mathbb{C}^k\subset\mathbb{C}^n$ 
    is an isotropic subspace with respect to the standard symplectic 
    form on~$\mathbb{C}^k$. 
    By Lemma~\ref{lem:LagrangianSubspace}, 
    there exists~$a\in U(k)$ such that
    $ap_2(\mathfrak{a})\subset \Rn{k}$.
    In particular, for
    $g:=\left(\begin{pmatrix}
        a &\\ &1
    \end{pmatrix}, 0\right)\in G$,
    we have
    \begin{equation*}
        Ad_g\mathfrak{a}\subset
        \left\{\left(
            \begin{pmatrix}
            0&  \\
                & D \\
            \end{pmatrix},
            \begin{pmatrix}
                y\\
                0
            \end{pmatrix}
            \right)
        ~:~
        D\in\mathfrak{u}(1)^{n-k},~y\in \Rn{k} \right\}
        =\mathfrak{a}_0^{n,k}.
    \end{equation*}
    Since both sides have dimension~$n$, equality holds.
\end{proof}

\subsection{$N=D_n$}
In this section, we show that some connected abelian closed subgroups
$\mathbb{A}_+^n, 
\mathbb{A}_+^{n,\mathrm{H}},
\mathbb{A}_+^{n,k}~(k=1,\ldots, n)$
of the automorphism group of the \Kahler~manifold
$N:=D_n$ given in Example~\ref{exm:ComplexSpaceForms}
induce effective proper Hamiltonian actions, 
and that the Hessian manifolds obtained 
as their quotient manifolds 
are isomorphic to standard models.

First, we recall the automorphism group of~$D_n$.
\begin{definition}
    For an invertible Hermitian matrix~$H$ 
    of size~$n+1$, we define
    $U(H):=\{U\in GL(n+1, \Cn{})~:~
    U^*HU=H\}$.
\end{definition}
In the following, 
we set~$H:=\begin{pmatrix}
        1_n& \\ & -1
    \end{pmatrix}$.
Then~$U(n,1)=U(H)$.
\begin{note}
    For the \Kahler~manifold~$N=D_n$ 
    (Example~\ref{exm:ComplexSpaceForms}), 
    the automorphism group is~$G\cong PU(H)$, 
    and~$PU(H)$ acts on~$N$ by fractional linear transformations.
    Namely, for~$g:=\left[\begin{pmatrix}
    a & b\\ c& d
    \end{pmatrix}\right]\in PU(H)$ and~$z\in D_n$,
    $g\cdot z=\frac{az+b}{cz+d}$.
\end{note}
We may replace the pseudo-Hermitian metric as appropriate in the discussion.
\begin{note}\label{not:NegativeModels}
    (1)~Let~$\alpha:=\frac{1}{\sqrt{2}}
    \begin{pmatrix}
        1&&1\\ 
        &\sqrt2 1_{n-1}&\\
        1&&-1
    \end{pmatrix}\in GL(n+1,\Cn{})$.
    Then the fractional linear transformation determined by~$\alpha$ 
    from~$D_n$ to~$D'_n:=\{(z^1,\ldots, z^n)\in \Cn{n}~:~2\Re z^1+
    |z^2|^2+\cdots+|z^n|^2<0\}$ is a biholomorphic map, and
    for the \Kahler~metric~$h_{\mathrm{std}}$ 
    induced on~$D'_n$ from the \Kahler~metric on~$D_n$, 
    $\Psi:=-4\log (-2\Re z^1-
    |z^2|^2-\cdots-|z^n|^2)$ is a \Kahler~potential on~$D'_n$.\\
    (2)~For~$H':=\begin{pmatrix}&&1\\ &1_{n-1}&\\1&&
    \end{pmatrix}\in GL(n+1,\Cn{})$, 
    $U(H')$ acts holomorphically on
    $D'_n$, and the automorphism group of~$D'_n$ as a \Kahler~manifold
    is given by~$PU(H')$.
\end{note}

\begin{definition}\label{def:NegativeAbelianSubgroups}
    We define~$n$-dimensional connected abelian
    closed subgroups
    $\mathbb{A}_+^n,
    \mathbb{A}_+^{n,\mathrm{H}},
    \mathbb{A}_+^{n,k}$ ($k=1,\ldots, n$)
    of~$PU(H)$ as follows:
    \begin{align*}
        \mathbb{A}_+^{n}
        &:=\left\{
            \left[
            \begin{pmatrix}
                a &\\
                 &1
            \end{pmatrix}
        \right]\in PU(H)~:~a\in U(1)^n
        \right\},\\
        \mathbb{A}_+^{n,\mathrm{H}}
        &:=\alpha\left\{
            \left[
                \begin{pmatrix}
                    b & &\\
                    & a & \\
                    & & b^{-1}
                \end{pmatrix}
            \right]\in PU(H')~:~
            a\in U(1)^{n-1},~b\in \Rn{}_{>0}
        \right\}\alpha^{-1},\\
        \mathbb{A}_+^{n,k}
        &:=\alpha\left\{
            \left[
            \begin{pmatrix}
                1 & b^T & 0 & -\frac{1}{2}\|b\|^2+\iu c\\
                & 1& 0& -b\\
                & &a& 0\\
                &&&1
            \end{pmatrix}
            \right]\in PU(H')~:~
            a\in U(1)^{n-k},~b\in\Rn{k-1},~c\in \Rn{}
        \right\}\alpha^{-1}.
    \end{align*}
    Then,~$\mathfrak{a}_+^n,
    \mathfrak{a}_+^{n,\mathrm{H}}
    ,\mathfrak{a}_+^{n,k}$ are given by:
    \begin{align*}
        \mathfrak{a}_+^n
        &=\left\{
            \left[
            \begin{pmatrix}
                X &\\
                 &0
            \end{pmatrix}
        \right]~:~X\in \mathfrak{u}(1)^n
        \right\},\\
        \mathfrak{a}_+^{n,\mathrm{H}}
        &=\alpha\left\{
            \left[
                \begin{pmatrix}
                    y & &\\
                    & X & \\
                    & & -y
                \end{pmatrix}
            \right]~:~
            X\in\mathfrak{u}(1)^{n-1},~y\in\Rn{}
        \right\}\alpha^{-1},\\
        \mathfrak{a}_+^{n,k}
        &=\alpha\left\{
            \left[
            \begin{pmatrix}
                0 & y^T & 0 & \iu u\\
                & 0& 0& -y\\
                & &X& 0\\
                &&&0
            \end{pmatrix}
            \right]~:~
            X\in \mathfrak{u}(1)^{n-k},~y\in\Rn{k-1},~u\in \Rn{}
        \right\}\alpha^{-1}.
    \end{align*}
\end{definition}
\begin{fact}[\cite{Olm90}]\label{fac:NegativeAbelianSubgroup}
    An~$n$-dimensional connected abelian Lie subgroup of~$PU(n,1)=PU(H)$ 
    is conjugate to one of the subgroups~$\mathbb{A}_+^n,
    \mathbb{A}_+^{n,\mathrm{H}}
    ,\mathbb{A}_+^{n,k}$ 
    defined in Definition~\ref{def:NegativeAbelianSubgroups}.
\end{fact}

In the following, 
we show that all the subgroups defined 
in Definition~\ref{def:NegativeAbelianSubgroups} 
give effective proper Hamiltonian actions 
(Proposition~\ref{prop:NegativeQuotient}).
\begin{lemma}\label{lem:NegativeQuotient}
    Let~$C<0$.\\
    (1)~Set~$N:=D_n$ with~$\frac{1}{|C|}h_{\mathrm{std}}$ 
    (Example~\ref{exm:ComplexSpaceForms})
    and~$\mathbb{A}:=\mathbb{A}_+^{n}\subset PU(H)$.
    Then the map~$\mu\colon N\rightarrow \mathfrak{a}^*\cong\Rn{n}$
    defined by
    \begin{equation*}
        \mu(z^1,\ldots,z^n):=-\frac{1}{|C|(1-\|z\|^2)}(|z^{1}|^2,
        \ldots,|z^n|^2)
    \end{equation*}
    is a moment map.
    Here, we identify~$\mathfrak{a}$ with~$\Rn{n}$
    via the basis
    \begin{equation*}
        \left(\frac{\iu}{2}E_{11},\ldots, \frac{\iu}{2}E_{nn}\right),
    \end{equation*}
    where~$E_{ii}$ denotes the matrix unit.\\
    (2)~Set~$N:=D'_n$ with~$\frac{1}{|C|}h_{\mathrm{std}}$ 
    (Note~\ref{not:NegativeModels})
    and~$\mathbb{A}:=\alpha^{-1}\mathbb{A}_+^{n,\mathrm{H}}\alpha\subset PU(H')$.
    Then the map~$\mu\colon N\rightarrow \mathfrak{a}^*\cong\Rn{n}$
    defined by
    \begin{equation*}
        \mu(z^1,\ldots,z^n):=-\frac{1}{|C|(-2\Re z^1-|z^2|^2-\cdots-|z^n|^2)}
        (2\Im z^1, |z^{2}|^2,
        \ldots,|z^n|^2)
    \end{equation*}
    is a moment map.
    Here, we identify~$\mathfrak{a}$ with~$\Rn{n}$
    via the basis
    \begin{equation*}
        \left(\frac{1}{2}(E_{11}-E_{(n+1)(n+1)}),
    \frac{\iu}{2}E_{22},\ldots, \frac{\iu}{2}E_{nn}\right),
    \end{equation*}
    where~$E_{ii}$ denotes the matrix unit.\\
    (3)~Let~$k=1,\ldots, n$.
    Set~$N:=D'_n$ with~$\frac{1}{|C|}h_{\mathrm{std}}$ 
    (Note~\ref{not:NegativeModels})
    and~$\mathbb{A}:=\alpha^{-1}\mathbb{A}_+^{n,k}\alpha\subset PU(H')$.
    Then the map~$\mu\colon N\rightarrow \mathfrak{a}^*\cong\Rn{n}$
    defined by
    \begin{equation*}
        \mu(z^1,\ldots,z^n):=-\frac{1}{|C|(-2\Re z^1-|z^2|^2-\cdots-|z^n|^2)}
        (1, 2\Im z^2,\ldots, 2\Im z^k,  |z^{k+1}|^2,
        \ldots,|z^n|^2)
    \end{equation*}
    is a moment map.
    Here, we identify~$\mathfrak{a}$ with~$\Rn{n}$
    via the basis
    \begin{align*}
        &\left(
    \frac{1}{2}\begin{pmatrix}
        0&0&0&1\\
         &0&0&0\\
         & &0&0\\
         & & &0
    \end{pmatrix},
    \frac{1}{2}\begin{pmatrix}
        0&e_2^T&0&0\\
         &0&0&-e_2\\
         & &0&0\\
         & & &0
    \end{pmatrix}
    \ldots,
    \frac{1}{2}\begin{pmatrix}
        0&e_k^T&0&0\\
         &0&0&-e_k\\
         & &0&0\\
         & & &0
    \end{pmatrix},\right.\\
    &\quad \left.\frac{\iu}{2}\begin{pmatrix}
        0&0&0&0\\
         &0&0&0\\
         & &E_{(k+1)(k+1)}&0\\
         & & &0
    \end{pmatrix},
    \ldots,
    \frac{\iu}{2}\begin{pmatrix}
        0&0&0&0\\
         &0&0&0\\
         & &E_{nn}&0\\
         & & &0
    \end{pmatrix}\right),
    \end{align*}
    where~$e_i$ denotes the standard basis of~$\Rn{n}$,
    and~$E_{ii}$ denotes the matrix unit.
\end{lemma}
\begin{proof}
    It suffices to consider the case~$C=-1$.
    For (1), (2) and (3),
    set~$M:=\mathcal{H}^{n}(1),
    \mathcal{H}^{n,\mathrm{H}}(1),
    \mathcal{H}^{n,k}(1)$ ($k=1,\ldots, n$)
    defined in Example~\ref{exm:StandardModels},
    respectively.
    Let~$N':=M\times \Rn{n}$
    be equipped with the \Kahler~structure
    in Fact~\ref{fac:TangentBundleKahler} and the
    effective proper Hamiltonian~$\mathbb{R}^n$-action in
    Proposition~\ref{prop:TangentMomentMap}.
    We denote by~$\Psi'$ and~$\mu'$ the \Kahler~potential and moment map,
    respectively, on~$N'$ defined in
    Example~\ref{exm:StandardModelsKahler}.
    We denote by~$(w^1,\ldots,w^n)$
    the complex chart on~$N'$ from
    Note~\ref{not:TangentHolomorphicChart}.

    (1)~
    By Lemma~\ref{lem:OpenDense},
    it suffices to show that the restriction of~$\mu$
    to~$N_\mathbb{A}^\circ=\{(z^1,\ldots,z^n)
    \in D_n~:~ z^1,\ldots, z^n\neq 0\}$
    is a moment map.
    We use Proposition~\ref{prop:MomentMaps}.

    Define~$F$ by
    \begin{equation*}
        N'\rightarrow N^\circ_\mathbb{A},~
        (w^1,\ldots, w^n)\mapsto \left(e^\frac{w^1}{2},\cdots
        ,e^{\frac{w^n}{2}}\right).
    \end{equation*}
    We also define
    the surjective homomorphism
    $\rho\colon \Rn{n}\rightarrow \mathbb{A}$ by
    \begin{equation*}
        \rho\colon \Rn{n}\rightarrow \mathbb{A},~
        (a^1,\ldots,a^n)\mapsto 
        \left[\begin{pmatrix}
            e^{\frac{\iu}{2}a^1}& & & \\
            & \ddots & & \\
            & &e^{\frac{\iu}{2}a^n} &\\
            & & &1
        \end{pmatrix}\right].
    \end{equation*}
    It therefore suffices to show that~$F$ is a surjective local
    diffeomorphism preserving the \Kahler~structures, and that
    $F$,~$\rho$,~$\mu'$ and~$\mu$ satisfy conditions~(i) and~(ii) of
    Proposition~\ref{prop:MomentMaps}.

    Conditions~(i) and~(ii) follow by direct computation.
    Moreover,~the \Kahler~potential
    $\Psi:=-4\log(1-\|z\|^2)$
    on~$N$ satisfies~$\Psi'\circ F-\Psi=0.$
    Hence, by Note~\ref{not:KahlerPotentialDifference},
    $F$ preserves the \Kahler~structures.
    The inverse of~$F$ is locally given by
    \begin{equation*}
        N^\circ_\mathbb{A}\rightarrow N',~
        (z^1,\ldots, z^n)\mapsto (2\log z^{1},\ldots, 2\log z^n),
    \end{equation*}
    where a branch of~$\log$ is chosen appropriately. 
    Therefore,~$F$ is a surjective local diffeomorphism.

    (2)~By Lemma~\ref{lem:OpenDense},
    it suffices to show that the restriction of~$\mu$
    to~$N_\mathbb{A}^\circ=\{(z^1,\ldots,z^n)
    \in D_n~:~ z^2,\ldots, z^n\neq 0\}$
    is a moment map.
    We use Proposition~\ref{prop:MomentMaps}.

    Define~$F$ by
    \begin{equation*}
        N'\rightarrow N^\circ_{\mathbb{A}},~
        (w^1,\ldots,w^n)\mapsto 
        \left(-\frac{1}{2}e^{-\iu w^1}, -\iu e^{\frac{w^2}{2}-\iu\frac{w^1}{2}},
        \ldots,-\iu e^{\frac{w^n}{2}-\iu\frac{w^1}{2}}\right).
    \end{equation*}
    We also define
    the surjective homomorphism
    $\rho\colon \Rn{n}\rightarrow \mathbb{A}$ by
     \begin{equation*}
        \rho\colon \Rn{n}\rightarrow \mathbb{A},~
        (a^1,\ldots,a^n)\mapsto 
        \left[\begin{pmatrix}
            e^{\frac{a^1}{2}}& & & &\\
            & e^{\frac{\iu}{2}a^2} & && \\
            & &\ddots & & \\
            & & &e^{\frac{\iu}{2}a^n} &\\
            & & & &e^{-\frac{a^1}{2}}
        \end{pmatrix}\right].
    \end{equation*}
    It therefore suffices to show that~$F$ is a surjective local
    diffeomorphism preserving the \Kahler~structures, and that
    $F$,~$\rho$,~$\mu'$ and~$\mu$ satisfy conditions~(i) and~(ii) of
    Proposition~\ref{prop:MomentMaps}.

    Conditions~(i) and~(ii) follow by direct computation.
    Moreover,~the \Kahler~potential
    $\Psi:=-4\log (-2\Re z^1-|z^2|^2-\cdots-|z^n|^2)$
    on~$N$ satisfies
    $\Psi'\circ F-\Psi=-4\Im w^1.$
    Hence, by Note~\ref{not:KahlerPotentialDifference},
    $F$ preserves the \Kahler~structures.
    The inverse of~$F$ is locally given by
    \begin{align*}
        N^\circ_\mathbb{A}&\rightarrow N',\\
        (z^1,\ldots, z^n)&\mapsto (\iu\log(-2z^1),
        2\log(\iu z^2)-\log(-2z^1)\ldots, 
        2\log(\iu z^n)-\log(-2z^1)),
    \end{align*}
    where a branch of~$\log$ is chosen appropriately. 
    Therefore,~$F$ is a surjective local diffeomorphism.
    
    (3)~By Lemma~\ref{lem:OpenDense},
    it suffices to show that the restriction of~$\mu$
    to~$N_\mathbb{A}^\circ=\{(z^1,\ldots,z^n)
    \in D_n~:~ z^{k+1},\ldots, z^n\neq 0\}$
    is a moment map.
    We use Proposition~\ref{prop:MomentMaps}.

    Define~$F$ by
    \begin{equation*}
        N'\rightarrow N^\circ_{\mathbb{A}},~
        (w^1,\ldots,w^n)\mapsto 
        \left(\frac{w^1}{2}+\frac{(w^2)^2}{8}+
        \cdots+\frac{(w^k)^2}{8},
        \iu\frac{w^2}{2},\ldots,\iu\frac{w^k}{2},
         e^\frac{w^{k+1}}{2},\cdots
        ,e^{\frac{w^n}{2}}\right).
    \end{equation*}
    We also define
    the surjective homomorphism
    $\rho\colon \Rn{n}\rightarrow \mathbb{A}$ by
     \begin{align*}
        \rho\colon \Rn{n}&\rightarrow \mathbb{A},\\
        (a^1,\ldots,a^n)&\mapsto 
        \left[\begin{pmatrix}
            1 &\frac{a^2}{2}&\cdots &\frac{a^k}{2} &0 &\cdots &0&
            -\frac{(a^2)^2+\cdots+(a^k)^2}{8}+\frac{\iu}{2}a^1\\
             & 1& && & & &-\frac{a^2}{2}\\
              & &\ddots && & & &\vdots\\
             & & &1 & & & &-\frac{a^k}{2}\\
             & &&  &e^{\frac{\iu}{2}a^{k+1}}& & &0\\
             & & &&  &\ddots& &\vdots\\
             & & & & & &e^{\frac{\iu}{2}a^n} &0\\
             & & & & & & &1
        \end{pmatrix}\right].
    \end{align*}
    It therefore suffices to show that~$F$ is a surjective local
    diffeomorphism preserving the \Kahler~structures, and that
    $F$,~$\rho$,~$\mu'$ and~$\mu$ satisfy conditions~(i) and~(ii) of
    Proposition~\ref{prop:MomentMaps}.

    Conditions~(i) and~(ii) follow by direct computation.
    Moreover,~the \Kahler~potential
    $\Psi:=-4\log (-2\Re z^1-|z^2|^2-\cdots-|z^n|^2)$
    on~$N$ satisfies
    $\Psi'\circ F-\Psi=0.$
    Hence, by Note~\ref{not:KahlerPotentialDifference},
    $F$ preserves the \Kahler~structures.
    The inverse of~$F$ is locally given by
    \begin{align*}
        N^\circ_\mathbb{A}&\rightarrow N',\\
        (z^1,\ldots, z^n)&\mapsto (2z^1+(z^2)^2+\cdots+(z^k)^2,
        -2\iu z^2,\ldots, -2\iu z^k,
        2\log z^{k+1},\ldots, 2\log z^n).
    \end{align*}
    where a branch of~$\log$ is chosen appropriately. 
    Therefore,~$F$ is a surjective local diffeomorphism.
\end{proof}
\begin{proposition}\label{prop:NegativeQuotient}
    Let~$C<0$ and~$k=1,\ldots, n$.
    Set~$N:=D_n$ with metric~$\frac{1}{|C|}h_{\mathrm{std}}$ 
    (Example~\ref{exm:ComplexSpaceForms}).\\
    (1)~$\mathbb{A}_+^n,
        \mathbb{A}_+^{n,\mathrm{H}}$ and
        $\mathbb{A}_+^{n,k}$ naturally induce Hamiltonian actions on~$N$.\\
    (2)~For~$\mathbb{A}:=\begin{cases}
        \mathbb{A}_+^n\\
        \mathbb{A}_+^{n,\mathrm{H}}\\
        \mathbb{A}_+^{n,k}
    \end{cases}$,
    we have~$N_\mathbb{A}^\circ/\mathbb{A}
    \cong \begin{cases}
        \mathcal{H}_+^n(-C)\\
        \mathcal{H}_+^{n,\mathrm{H}}(-C)\\
        \mathcal{H}_+^{n,k}(-C)
    \end{cases}$, respectively, as Hessian manifolds. 
\end{proposition}
\begin{proof}
    It suffices to consider the case~$C=-1$.
    Moreover, 
    in the cases~$\mathbb{A}=\mathbb{A}_+^{n,\mathrm{H}},
    \mathbb{A}_+^{n,k}$, 
    we may take~$N=D'_n$ (Note~\ref{not:NegativeModels}) 
    and $\mathbb{A}=\alpha\mathbb{A}_+^{n,\mathrm{H}}\alpha^{-1}, 
    \alpha\mathbb{A}_+^{n,k}\alpha^{-1}$, respectively.
    For~$\mathbb{A}=\mathbb{A}_+^n, \mathbb{A}_+^{n,\mathrm{H}}, 
    \mathbb{A}_+^{n,k}$, 
    we use~$F, \rho, \mu, \mu'$ defined 
    in Lemma~\ref{lem:NegativeQuotient}~(1), (2), (3) and their proofs, 
    respectively.

    (1) follows immediately from Lemma~\ref{lem:NegativeQuotient}.
    We prove (2). 
    For~$\mathbb{A}=\mathbb{A}_+^n, 
    \mathbb{A}_+^{n,\mathrm{H}}, \mathbb{A}_+^{n,k}$, 
    let~$M:=\mathcal{H}^{n}(1), \mathcal{H}^{n,\mathrm{H}}(1), 
    \mathcal{H}^{n,k}(1)$, respectively. Then,
    by Proposition~\ref{prop:QuotientMorphism}, 
    it suffices to show
    that the local diffeomorphism
    $f\colon \mathcal{R}^{n,k}(0)\rightarrow N_\mathbb{A}^\circ/\mathbb{A}$
    induced by~$F$ is bijective.

    Surjectivity follows from the surjectivity of~$F$.
    To prove injectivity, by Proposition~\ref{prop:QuotientMorphism}, 
    we have
    $\pi(F(p,0))=f(p)$.
    Thus, for arbitrary $p,q\in M$, if
    $f(p)=f(q)$, then
    $\pi(F(p,0))=\pi(F(q,0))$.
    Since~$\mu$ is~$\mathbb{A}$-invariant, this implies
    $\mu(F(p,0))=\mu(F(q,0)).$
    On the other hand, the map
    $M\rightarrow
    \mathfrak{a}^*,~p\mapsto \mu(F(p,0))$
    is given by
    \begin{align*}
        &(\theta^1,\ldots,\theta^n)\\
        &\mapsto
        \begin{cases}
            -\left(\theta^1,\ldots,\theta^k,
            e^{\theta^{k+1}},\ldots, e^{\theta^n} \right)
            & (\mathbb{A}=\mathbb{A}_+^n)\\
            \frac{1}{\cos\theta^1-e^{\theta^2}
            \cdots-e^{\theta^n}}
            \left(\sin \theta^1, e^{\theta^2},\ldots,e^{\theta^2}\right)
            &(\mathbb{A}=\mathbb{A}_+^{n,\mathrm{H}})\\
            \frac{1}{-\theta^1-
            \frac{1}{2}(\theta^2)^2-\cdots- \frac{1}{2}(\theta^k)^2
            -e^{\theta^{k+1}}-\cdots-e^{\theta^n}}\left(1, \theta^2,\ldots, \theta^k,
            e^{\theta^{k+1}},\ldots, e^{\theta^n}\right)
            &(\mathbb{A}=\mathbb{A}_+^{n,k})
        \end{cases}
    \end{align*}
    in each case, which is injective by a direct computation. 
    Hence~$f$ is injective. Therefore,~$f$ is bijective.
\end{proof}

\subsection{Correspondence and Local embedding}\label{sec:LocalEmbedding}
In this section, we prove Proposition~\ref{prop:LocalHessianEmbedding} 
by using the construction of Hessian manifolds 
as quotients of \Kahler~manifolds.
Furthermore, we prove Theorem~\ref{mthm:Correspondence}.

\begin{definition}
    For~$C\in\Rn{}$, 
    we define a \Kahler~manifold~$N^n(C)$ by
    \begin{equation*}
        N^n(C):=\begin{cases}
        \mathbb{C}P^n &(C>0)\\
        \mathbb{C}^n  &(C=0)\\
        D_n  & (C<0)
    \end{cases}
    \quad \mathrm{with~metric}~
    \begin{cases}
        \frac{1}{C}h_{\mathrm{std}} &(C>0)\\
        h_{\mathrm{std}}  &(C=0)\\
        \frac{1}{|C|}h_{\mathrm{std}}  & (C<0)
    \end{cases}.
    \end{equation*}
\end{definition}
The above Definition is motivated by the
following Facts~\ref{fac:ComplexSpaceForms}
and~\ref{fac:ComplexSpaceFormsEmbedding}.
\begin{fact}[\cite{KN69}]\label{fac:ComplexSpaceForms}
    For~$C\in\Rn{}$,
    $N^n(C)$ has constant holomorphic sectional curvature~$C$.
\end{fact}
See \cite{KN69} for the definition of the holomorphic sectional curvature.
\begin{fact}[\cite{KN69}]\label{fac:ComplexSpaceFormsEmbedding}
    If a~$2n$-dimensional 
    \Kahler~manifold has constant holomorphic sectional 
    curvature~$C\in \Rn{}$, 
    then there exists an open embedding 
    from a neighborhood of any point into~$N^n(C)$.
\end{fact}
\begin{fact}[\cite{Kob72}]\label{fac:HolomorphicKilling}
    For a \Kahler~manifold~$N$, 
    the Lie algebra~$\mathfrak{g}$ 
    of its automorphism group~$G:=\Aut(N)$ 
    is naturally identified with the Lie algebra 
    of complete holomorphic Killing vector fields on~$N$.
    In particular, if~$N$ is complete, 
    then every holomorphic Killing vector field is complete, 
    and~$\mathfrak{g}$ coincides with 
    the Lie algebra of all holomorphic Killing vector fields.
\end{fact}
\begin{fact}[\cite{Nom60}]\label{fac:KillingExtention}
    Any locally defined Killing vector field on
    a connected, simply connected, complete and analytic 
    Riemannian manifold
    extends uniquely to a global Killing vector field.
\end{fact}
\begin{lemma}\label{lem:KillingExtention}
    Let 
    $(N, I,h,\omega)$ be a connected, simply connected,
    complete, and analytic \Kahler~manifold.
    Then any holomorphic Killing vector field defined on
    a connected open subset 
    $W\subset N$ extends globally.
\end{lemma}
\begin{proof}
    Let 
    $X$ be a holomorphic Killing vector field defined on~$W$.
    By Fact~\ref{fac:KillingExtention}, it extends to a global
    Killing vector field 
    $\tilde{X}$.
    It suffices to show that 
    $\tilde{X}$ is holomorphic,
    that is, 
    $\mathcal{L}_{\tilde{X}}I=0$.
    In particular, by assumption, we have
    $\mathcal{L}_{\tilde{X}}I=0$ on~$W$.
    Hence, it suffices to show that
    $\nabla^h(\mathcal{L}_{\tilde{X}}I)=0$ on~$N$, where 
    $\nabla^h$
    is the Levi-Civita connection of~$h$.

    By Fact~\ref{fac:KahlerLeviCivita}, for any vector fields
    $Y,Z$ on 
    $N$, we have
    \begin{equation*}
        0=(\nabla^h_YI)Z=\nabla^h_Y(IZ)-I\nabla^h_YZ.
    \end{equation*}
    Moreover, since~$\tilde{X}$ is a Killing vector field,
    we have
    \begin{equation*}
        0=(\mathcal{L}_{\tilde{X}}\nabla^h)_YZ
        =[\tilde{X},\nabla^h_YZ]
        -\nabla^h_{[\tilde{X}, Y]}Z
        -\nabla^h_Y[\tilde{X},Z].
    \end{equation*}
    Therefore,
    \begin{align*}
        (\nabla^h_Y(\mathcal{L}_{\tilde{X}}I))Z
        &=\nabla^h_Y[\tilde{X},IZ]
        -I\nabla^h_Y[\tilde{X},Z]
        -[\tilde{X},I\nabla^h_YZ]
        +I[\tilde{X},\nabla^h_YZ]\\
        &=\nabla^h_Y[\tilde{X},IZ]
        -I\nabla^h_Y[\tilde{X},Z]
        -I\nabla^h_{[\tilde{X},Y]}Z\\
        &\quad -\nabla^h_Y[\tilde{X},IZ]
        +I\nabla^h_{[\tilde{X},Y]}Z
        +I\nabla^h_Y[\tilde{X},Z]=0.
    \end{align*}
\end{proof}
\begin{lemma}\label{lem:KahlerLocalEmbedding}
    Let~$M$ be an~$n$-dimensional Hessian manifold 
    with constant Hessian sectional curvature~$c\in\Rn{}$, 
    and let~$N:=N^n(-c)$.
    Then, for any point~$p\in M$, 
    there exists an~$n$-dimensional abelian subgroup
    ~$\mathbb{A}\subset G:=\Aut(N)$ satisfying the following conditions:\\
    (1)~$\mathbb{A}$ is conjugate to one of the subgroups given by
    $\begin{cases}
            \mathbb{A}_-^n &(c<0)\\
            \mathbb{A}_0^{n,k}~(k=0,\ldots, n)&(c=0)\\
            \mathbb{A}_+^{n},
            \mathbb{A}_+^{n,\mathrm{H}},
            \mathbb{A}_+^{n,k}~(k=1,\ldots,n)& (c>0)
        \end{cases}.$\\
(2)~The natural action of~$\mathbb{A}$ on~$N$ is 
effective, proper and Hamiltonian.\\
(3)~There exists an open embedding~$U\rightarrow N^\circ_{\mathbb{A}}/\mathbb{A}$
    from an open neighborhood~$U\subset M$ of~$p$ 
    to~$N^\circ_{\mathbb{A}}/\mathbb{A}=N^\circ_\mathbb{A}/\mathbb{A}$ 
    that preserves the Hessian structures.
\end{lemma}
\begin{proof}
    Since (1)$\Rightarrow$(2) follows from 
    Proposition~\ref{prop:Section4} (2), 
    it suffices to prove (1) and (3).
    Fix a point~$p\in M$.
    Since the argument is local, 
    we may assume that~$M$ is a Hessian domain,
    and hence~$M$ is an open subset of~$\Rn{n}$.
    We also assume that~$M\times \Rn{n}$ 
    is equipped with the \Kahler~structure 
    in Fact~\ref{fac:TangentBundleKahler} 
    and the effective proper Hamiltonian
    $\Rn{n}$-action in Proposition~\ref{prop:TangentMomentMap}.
    By Fact~\ref{fac:TangentCurvature}, 
    $M\times\Rn{n}$ has constant holomorphic sectional curvature~$-c$. 
    Thus, by Fact~\ref{fac:ComplexSpaceFormsEmbedding}, 
    there exists an open embedding~$F: W\rightarrow W'\subset N$ 
    preserving the \Kahler~structures, 
    where~$W$ is an open neighborhood of~$(p,0)\in M\times \Rn{n}$.
    We may further assume that~$W=U\times T$ 
    for some open subsets~$U\subset M$ and~$T\subset \Rn{n}$.

    First, we use~$F$ to define an~$n$-dimensional 
    abelian subgroup~$\mathbb{A}\subset G$ 
    from the~$\Rn{n}$-action on~$M\times \Rn{n}$.
    For an arbitrary~$X\in \Rn{n}$, 
    regard~$X$ as a parallel vector field on~$\Rn{n}$. 
    Then the vector field~$X^V:=\pr_2^*X$ on~$M\times \Rn{n}$, 
    defined in Fact~\ref{fac:TangentBundleKahler}, 
    coincides with the holomorphic Killing vector field 
    generated by~$X\in\Rn{n}$ with respect to 
    the~$\Rn{n}$-action on $M\times \Rn{n}$.
    By Lemma~\ref{lem:KillingExtention}, 
    the holomorphic Killing vector field~$F_*X^V$ on~$W'$ extends 
    to a global holomorphic Killing vector field~$\widetilde{F_*X^V}$ 
    on~$N$.
    Hence, by Fact~\ref{fac:HolomorphicKilling}, 
    $\widetilde{F_*X^V}$ is an element of the Lie algebra~$\mathfrak{g}$ 
    of~$G$. Moreover,
    $\mathfrak{a}:=\{
        \widetilde{F_*X^V}\in\mathfrak{g}~:~ X\in\Rn{n}
    \}$
    is an~$n$-dimensional abelian subalgebra, 
    which determines a subgroup~$\mathbb{A}\subset G$.

    We show that~$\mathbb{A}$ satisfies condition (1).
    By Proposition~\ref{prop:TangentMomentMap}, the 
    $\Rn{n}$-action
    on 
    $M\times \Rn{n}$ generates an isotropic subspace at each point.
    Hence, 
    $\mathbb{A}$ also generates an isotropic subspace at each
    point of 
    $W'\subset N$.
    Thus, the claim follows from Proposition~\ref{prop:Section4} (1).
    By adjusting 
    $F$ by an automorphism of 
    $N$, we may assume that
    $\mathbb{A}$ coincides with one of the subgroups, and hence (1)
    follows.
    In this case, by the definition of~$\mathfrak{a}$, we obtain
    a Lie algebra isomorphism
    $R\colon \Rn{n}\simarrow \mathfrak{a},~
    X\mapsto\widetilde{F_*X^V}$,
    and 
    $F_*X^V=(RX)^V$ holds for any 
    $X\in\Rn{n}$.
    Moreover, since the map is defined on 
    $\Rn{n}$,
    it determines a Lie group homomorphism
    $\rho\colon \Rn{n}\rightarrow \mathbb{A}$.
    Since 
    $\mathbb{A}$ is connected and abelian, 
    $\rho$ is surjective.
    
    Next, we show that condition (3) holds.
    For this purpose, it suffices to prove the following claim,
    from which (3) follows by Proposition~\ref{prop:QuotientMorphism}:\\
    \textbf{Claim.}\\
    The smooth map
    $\tilde{F}\colon U\times \Rn{n}\rightarrow N^\circ_{\mathbb{A}},~
    (p, a)\mapsto \rho(a)\cdot F(p,0)$
    is a local diffeomorphism preserving the \Kahler~structure and
    satisfies
    \begin{equation}\label{eq:equivariantness}
        \forall a\in\Rn{n},~\tilde{F}\circ a=\rho(a)\circ \tilde{F}.
    \end{equation}

    We prove the claim.
    The smoothness is clear.
    The equality \eqref{eq:equivariantness} follows immediately
    from a direct computation.
    It remains to show that 
    $\tilde{F}$ is a local diffeomorphism
    preserving the \Kahler~structure.
    By the 
    $\Rn{n}$- and 
    $\mathbb{A}$-invariance of the 
    \Kahler~structures on 
    $M\times \Rn{n}$ and 
    $N$, respectively,
    together with \eqref{eq:equivariantness}, it suffices to show
    that 
    $\tilde{F}$ is a local diffeomorphism preserving the
    \Kahler~structure on 
    $W\subset U\times \Rn{n}$.
    In particular, it suffices to show that 
    $\tilde{F}|_W=F$.
    First, by definition, we have 
    $\tilde{F}(p,0)=F(p,0)$
    for any 
    $p\in U$.
    Moreover, a direct computation shows that
    $\tilde{F}_*X^V_q=(RX)^V_{F(q)}$
    for any 
    $q\in U\times \Rn{n}$ and 
    $X\in\Rn{n}$.
    Hence, by the uniqueness of integral curves of vector fields,
    we obtain 
    $\tilde{F}|_W=F$.
\end{proof}

We prove Proposition~\ref{prop:LocalHessianEmbedding}, 
which will be used to prove Theorem~\ref{mthm:LocalEmbedding}.
\begin{proof}[Proof of Proposition~\ref{prop:LocalHessianEmbedding}]
    By Lemma~\ref{lem:KahlerLocalEmbedding}, 
    for any~$p\in M$, there exists a local open embedding into 
    a quotient manifold~$N^\circ_{\mathbb{A}}/\mathbb{A}$
    around~$p$ for a suitable choice of~$N$ and~$\mathbb{A}$. 
    By Proposition~\ref{prop:Section4}, 
    this quotient manifold coincides with one of the standard models.
\end{proof}
We now prove Theorem~\ref{mthm:Correspondence}.
\begin{proof}[Proof of Theorem~\ref{mthm:Correspondence}]
    The first assertion follows from Proposition~\ref{prop:Section4}.
    The second assertion follows from the fact 
    that the local diffeomorphism to the standard models 
    given by Proposition~\ref{prop:LocalHessianEmbedding} 
    extends to an isomorphism by 
    Corollary~\ref{cor:ConvexMaximalIsomorphism}.
\end{proof}

\bibliographystyle{amsplain}
\bibliography{masterthesis-bibliography}

@book{AB15,
  title={Lie groups and geometric aspects of isometric actions},
  author={Alexandrino, M. M. and Bettiol, R. G.},
  volume={82},
  year={2015},
  publisher={Springer}
}

@article{Abr03,
  title={K{\"a}hler Geometry of Toric Manifolds in Symplectic},
  author={Abreu, M.},
  journal={Symplectic and Contact Topology: Interactions and Perspectives: Interactions and Perspectives},
  volume={35},
  pages={1},
  year={2003}
}

@article{Cas96,
  title={The {$2 d+ 4$} simple quadratic natural exponential families on 
  $\mathbf{R}^d$},
  author={Casalis, M.},
  journal={The Annals of Statistics},
  volume={24},
  number={4},
  pages={1828--1854},
  year={1996},
  publisher={Institute of Mathematical Statistics}
}

@article{FK13,
  title={Hessian manifolds of nonpositive constant Hessian sectional curvature},
  author={Furuhata, H. and Kurose, T.},
  journal={Tohoku Mathematical Journal, Second Series},
  volume={65},
  number={1},
  pages={31--42},
  year={2013},
  publisher={Mathematical Institute, Tohoku University}
}

@book{GGK02,
  title={Moment maps, cobordisms, and Hamiltonian group actions},
  author={Guillemin, V. and Ginzburg, V. and Karshon, Y.},
  number={98},
  year={2002},
  publisher={American Mathematical Soc.}
}

@article{Gui94,
  title={K{\"a}hler structures on toric varieties},
  author={Guillemin,V.},
  journal={Journal of differential geometry},
  volume={40},
  number={2},
  pages={285--309},
  year={1994}
}

@book{KN63,
  author		= "Kobayashi, S. and Nomizu, K. ",
  title			= "Foundations of {D}ifferential {G}eometry",
  volume		= "1",
  address		= "New York",
  publisher		= "John Wiley \& Sons",
  year			= "1963"
}

@book{KN69,
  author		= "Kobayashi, S. and Nomizu, K. ",
  title			= "Foundations of {D}ifferential {G}eometry",
  volume		= "2",
  address		= "New York",
  publisher		= "John Wiley \& Sons",
  year			= "1969"
}

@book{Kob72,
  author		= "Kobayashi, S.",
  title			= "Transformation {G}roups in {D}ifferential {G}eometry",
  address		= "Berlin",
  publisher		= "Springer-Verlag",
  year			= "1972"
}

@book{Lee18,
  author		= "Lee, J. M.",
  title			= "Introduction to {R}iemannian {M}anifolds",
  volume		= "2nd ed.",
  address		= "New York",
  publisher		= "Springer",
  year			= "2018"
}

@misc{Mol21,
      title={Kahler toric manifolds from dually flat spaces}, 
      author={Molitor, M.},
      year={2021},
      eprint={2109.04839},
      archivePrefix={arXiv},
      primaryClass={math.DG},
      url={https://arxiv.org/abs/2109.04839}, 
}

@article{Nom60,
  title={On local and global existence of Killing vector fields},
  author={Nomizu, K.},
  journal={Annals of Mathematics},
  volume={72},
  number={1},
  pages={105--120},
  year={1960}
}

@article{Olm90,
  title="Maximal abelian subalgebras of pseudounitary Lie algebras",
  author="del Olmo, M. A. and Rodriguez, M. A. and Winternitz, P. and 
  Zassenhaus, H.",
  journal="Linear Algebra and its Applications",
  volume="135",
  pages={79--151},
  year={1990},
  publisher={Elsevier}
}

@book{Pet06,
  title={Riemannian geometry},
  author={Petersen, P.},
  year={2006},
  publisher={Springer}
}

@article{Shi76,
  author		= "Shima, H.",
  title			= "On certain locally flat homogeneous manifolds 
    of solvable {L}ie groups",
  journal		= "Osaka Journal of Mathematics",
  volume		= "13",
  pages			= "213--229",
  year			= "1976",
  doi			= ""
}

@article{Shi95,
  title={Hessian manifolds of constant Hessian sectional curvature},
  author={Shima, H.},
  journal={Journal of the Mathematical Society of Japan},
  volume={47},
  number={4},
  pages={735--753},
  year={1995}
}

@book{Shi07,
  title={The geometry of Hessian structures},
  author={Shima, H.},
  year={2007},
  publisher={World Scientific}
}

@book{Sil08,
  title={Lectures on symplectic geometry},
  author={Cannas da Silva, A.},
  volume={1764},
  year={2008},
  publisher={Springer}
}

@book{Tu17,
  author		= "Tu,  L. W.",
  title			= "Differential geometry: 
  connections, curvature, and characteristic classes",
  publisher		= "Springer",
  year			= "2017"
}
\end{document}